\documentclass{amsart}
\usepackage{amsmath,amstext,amssymb,amsfonts,amscd,graphicx}
\newtheorem{thm}{Theorem}[section]

\newtheorem{lem}[thm]{Lemma}

\theoremstyle{definition}

\theoremstyle{remark}
\newtheorem{rem}[thm]{Remark}
\numberwithin{equation}{section}
\newcommand{\norm}[1]{\left\Vert#1\right\Vert}

\newcommand{\normXs}[1]{\left\Vert#1\right\Vert_{X^{s,b}}}
\newcommand{\normXk}[1]{\left\Vert#1\right\Vert_{X^{k,b_1}}}
\newcommand{\normXl}[1]{\left\Vert#1\right\Vert_{X^{l,b}_{\pm}}}

\newcommand{\bracs}[1]{\langle#1\rangle}

\newcommand{\wh}[1]{\widehat{#1}}
\newcommand{\hv}{\widehat{v}}
\newcommand{\lr}[1]{\langle #1 \rangle}
\newcommand{\Se}{S_{\varepsilon}}
\newcommand{\We}{W_{\varepsilon\pm}}
\newcommand{\Le}{\Delta_{\varepsilon}}

\newcommand{\Phie}[1]{\phi_{\varepsilon}(#1)}

\newcommand{\Phieo}{\phi_{\varepsilon}(\xi_1)}
\newcommand{\Phiew}{\phi_{\varepsilon}(\xi_2)}

\title{On one dimensional Quantum Zakharov system}
\author{Jin-Cheng Jiang}
\address{Department of Mathematics, National Tsing Hua University, Hsinchu, Taiwan 30013, R.O.C.}
\email{jcjiang@math.nthu.edu.tw}

\author{Chi-Kun Lin}
\address{Department of Mathematical Sciences, Xi'an Jiaotong-Liverpool University, SIP. Suzhou, JiangSu, 215123, P.R. China}
\email{Andrew.Lin@xjtlu.edu.cn}

\author{Shuanglin Shao}
\address{Department of Mathematics, University of Kansas, Lawrence, KS  66045, USA}
\email{slshao@math.ku.edu}


\begin{document}

\begin{abstract}
In this paper, we discuss the properties of one dimensional quantum Zakharov
system which describes the nonlinear interaction between the
quantum Langmuir and quantum ion-acoustic waves.
The system~\eqref{E:Zakharov-1}-\eqref{E:Zakharov-2} with initial data
$(E(0),n(0),\partial_t n(0))\in H^k\bigoplus H^l\bigoplus H^{l-2}$
is local well-posedness in low regularity spaces (see Theorem~\ref{T:main-theorem} and Figure~\ref{F:klm}).
Especially, the low regularity result for $k$ satisfies $-3/4<k\leq -1/4$
is obtained by using the key observation that the convoluted phase function is convex and
careful bilinear analysis. The result can not be obtained by  using only Strichartz
inequalities for "Schr\"{o}dinger" waves.

\end{abstract}

\maketitle

\section{Introduction}

The one-dimensional quantum Zakharov equations with initial data read
\begin{subequations}
\begin{align}
\label{E:Zakharov-1} & i\frac{\partial E}{\partial t}+\frac{\partial^2 E}{\partial x^2}-\varepsilon^2\frac{\partial^4 E}{\partial x^4}=nE,\\
\label{E:Zakharov-2} & \frac{\partial^2 n}{\partial t^2}-\frac{\partial^2 n}{\partial x^2}+\varepsilon^2\frac{\partial^4 n}{\partial x^4}=\frac{\partial^2 |E|^2}{\partial x^2},\\
\label{E:Zakharov-3} &
E(0,x)=E_0(x),\quad n(0,x)=n_0(x),\quad \frac{\partial n}{\partial t}(0,x)=n_1(x),
\end{align}
\end{subequations}
where the complex valued function $E=E(t,x)$ is the envelope electric field and the real valued function $n=n(t,x)$ is the plasma density fluctuation (measured from its equilibrium value). They are defined in
$\mathbb{R}^{+}_{t}\times\mathbb{R}_{x}$.
We assume $E_0\in H^k(\mathbb{R}),
 n_0\in H^l(\mathbb{R})$ and $(-\Delta+\varepsilon^2\Delta^2)^{-1/2}n_1\in H^l(\mathbb{R})$ for the study of local well-posedness.
The dimensionless quantum parameter
\begin{equation}\label{parameter}
\varepsilon = \frac{\hbar \omega_i}{\kappa_B T_e}
\end{equation}
is the ratio between the ion plasmon energy and the electron thermal energy, where $\hbar$ is Planck's constant divided by $2\pi$, $\omega_i$ is the ion plasma frequency, $\kappa_B$ is the Boltzmann constant and $T_e$ is the electron fluid temperature.
The quantum Zakharov equations are obtained to describe the nonlinear interaction between high-frequency quantum Langmuir waves and the low-frequency quantum ion-acoustic waves~\cite{Ha11, HS09}. The formal classical limit $\varepsilon \to 0$ yields the original Zakharov equations:
\begin{subequations}
\begin{align}
\label{E:Zakharov-1a} & i\frac{\partial E}{\partial t}+\frac{\partial^2 E}{\partial x^2}=nE\;,\\
\label{E:Zakharov-2a} & \frac{\partial^2 n}{\partial t^2}-\frac{\partial ^2 n}{\partial x^2}=\frac{\partial^2 |E|^2}{\partial x^2},\\
\label{E:Zakharov-3a} &
E(0,x)=E_0(x),\quad n(0,x)=n_0(x),\quad
\frac{\partial n}{\partial t}(0,x)=n_1(x),
\end{align}
\end{subequations}
which are one of the most important models in plasma physics \cite{SS99, Zak72}.
They describe the interaction between high-frequency Langmuir waves and low-frequency ion-acoustic waves. For the adiabatic limit of the Zakharov equations (\ref{E:Zakharov-1a})--(\ref{E:Zakharov-2a}), one neglects the second order time derivative of the density fluctuation, $\frac{\partial^2 n}{\partial t^2}\approx 0$, then $\frac{\partial^2}{\partial x^2}(n+|E|^2)=0$ implies $n=-|E|^2$ and the resulting equation is the cubic nonlinear Schr\"odinger equation
\begin{equation}\label{NLS}
i\frac{\partial E}{\partial t}+|E|^2 E +\frac{\partial^2 E}{\partial x^2} =0\,,
\end{equation}
which is known to be completely integrable and is one of most important nonlinear partial differential equations. However, for the quantum Zakharov equations (\ref{E:Zakharov-1})--(\ref{E:Zakharov-2}), the adiabatic limit will be
\begin{subequations}
\begin{align}
\label{E:Zakharov-1b} & i\frac{\partial E}{\partial t} +|E|^2 E+\frac{\partial^2 E}{\partial x^2}=\varepsilon^2\Big(\frac{\partial^4 E}{\partial x^4} +E\frac{\partial^2 n}{\partial x^2}\Big)\;,\\
\label{E:Zakharov-2b} & - n+\varepsilon^2\frac{\partial^2 n}{\partial x^2}= |E|^2\;,
\end{align}
\end{subequations}
If we further take the limit $\varepsilon \to 0$, the semiclassical limit, then $n=-|E|^2$ and the cubic nonlinear Schr\"odinger equation (\ref{NLS}) will be recovered. Thus it is natural to consider
\begin{equation}\label{pNLS}
i\frac{\partial E}{\partial t}+|E|^2 E +\frac{\partial^2 E}{\partial x^2} =\varepsilon^2\Big(\frac{\partial^4 E}{\partial x^4} -E\frac{\partial^2 |E|^2}{\partial x^2}\Big)\,,
\end{equation}
as the quantum perturbation of the cubic nonlinear Schr\"odinger equation.

The main purpose of this paper is to study the local well-posedness of one dimensional quantum Zakharov system (\ref{E:Zakharov-1})--(\ref{E:Zakharov-2}) with low regularity initial data.  The local well-posedness of Cauchy problem for the Zakharov system
in Euclidean space has been
extensively studied. We do not intend to list all the references, one can
see for example~\cite{OT92,BC96,GTV97,BKS00,CHT08,BHHT09,BH10} and references therein.
Unlike the Zakharov system, there are only a few well-posedness results for the quantum Zakharov system. The current results are mainly focused on
higher regularity spaces, for example Guo, Zhang and Guo~\cite{GZG12} showed that
$d$-dimensional ($d=1,2,3$) quantum Zakharov system is global well-posedness in
$H^k\bigoplus H^{k-1}\bigoplus H^{k-3}$ for integer $k\geq 2$. Since $k=2$ is the 
energy space,  the global well-posedness  is a consequence of the local well-posedness and
the conservation law of quantum Zakharov system (see section 2).
In general, the well-posedness results of the Zakharov system as well as other dispersive equations with low regularity initial data
can be established by the Strichartz inequalities. The key step is to derive non-linear estimates
by extensive use of Strichartz inequalities which has its origin from Bourgain~\cite{Bo93}.  For one dimensional Zakharov system, Ginibre, Tsutsumi and Velo~\cite{GTV97} established the local well-posedness result in low regularity spaces by adapting a
method first proposed by Kenig, Ponce and Vega~\cite{KPV96} to treat the Korteweg-de Vries equation which is a variant of Bourgain's method. Their method does not use Strichartz inequalities in the derivation of non-linear estimates, and relies instead on using Schwarz inequality cleverly followed by a direct estimation.

Similar to the method developed by Ginibre, Tsutumi and Velo~\cite{GTV97} in studying
one spatial dimensional Zakharov system, we combine the Strichartz
and Schwarz inequalities to estimate the non-linear interactions.
The challenge in the study of the quantum Zakharov system is that the interactions
of the non-linear part are much more complicated than that of
the Zakharov system due to the appearance of the fourth order terms as well as
the quantum parameter $\varepsilon$ in~\eqref{E:Zakharov-1} and~\eqref{E:Zakharov-2}.
When two waves are close enough, their interactions can be treated by Strichartz
inequalities as in~\cite{GTV97}.  However, when two waves are away from each other,
we have to use the Schwarz inequality instead.
In~~\cite{GTV97}, this part is not complicated and can be overcome 
by the change of variable argument. In our case, the bi-harmonic operator
prohibits us to apply that method.
The miracle here is that we can make use of the key observation that the convoluted phase function
of ``waves" is convex to get the quantitative estimates which describe the
separation of waves through the bilinear analysis. It is worth noting that the lower regularity
result for $k$ satisfies $-3/4<k\leq -1/4$
can not be obtained by Strichartz inequalities.

These quantitative  estimates allow us to get the well-posedness in low regularity spaces.
We believe that this ingredient will be the key for studying the quantum Zakharov system in
higher dimensional spaces and the other couple dispersive systems.

The main result of this paper is the following.
\begin{thm}\label{T:main-theorem}
For any fixed $0<\varepsilon\leq 1$,
the quantum Zakharov system~${\rm (\ref{E:Zakharov-1})}$-${\rm (\ref{E:Zakharov-2})}$ with initial data
$(E_0,n_0,n_1)\in H^k\oplus H^l\oplus H^{l-2}$ is locally well-posed provided $(k,l)$ is in
the set $\mathbb{A}$ defined by
\begin{equation}\label{E:A-set-2}
\begin{split}
\mathbb{A}&=\left\{(k,l)|-\frac{3}{2}<k-l<\frac{3}{2}\;,\;k\geq 0 \right\}\\
& \cup \left\{(k,l)|2k-l>-\frac{3}{2}\;,\;k+l>-\frac{3}{2}\;,\; -\frac{3}{4} <k<0\right\}.
\end{split}
\end{equation}
Also see Figure~\ref{F:klm}.
\end{thm}

\begin{rem}
The condition $\varepsilon\leq 1$ is for the convenience of discussion and the case
$\varepsilon\rightarrow 0$ is more interesting.
The lowest regularity we obtained here is the pair $(k,l)$ which is close to ${\rm C}=(-3/4,-3/4)$.
It is not clear whether the pair $(-3/4,-3/4)$ is optimal or not.
\end{rem}

\begin{figure}
\includegraphics[scale=0.5]{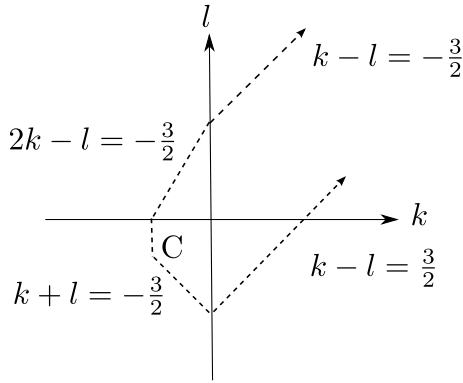}
\caption{Range of $k,l$}
\label{F:klm}
\end{figure}

\begin{rem}
 We can write the power $\frac{3}{2}$ as $\frac{4-1}{2}$.  Heuristically, $4$ comes from fourth order term and $1$ is necessary for
 non-linear estimates, while $2$ in the denominator is due to us working on $L^2$ based spaces.
\end{rem}

\begin{rem}
The dependence of time interval for well-posedness on $\varepsilon$ can be tracked explicitly as the $C(\varepsilon)$ in lemma
~\ref{L:bilinear-1} and~\ref{L:bilinear-2} are of order $\varepsilon^{-1},\varepsilon^{-2}$ respectively. They are from the estimates in
section 4,5,6 where $C(\varepsilon)$ are of order $\varepsilon^{-1}$ in lemma~\ref{L:C_1_bounded},~\ref{L:C_2_bounded} and
~\ref{L:C_3_bounded}, of order $\varepsilon^{-1/4}$ in lemma~\ref{L:key-lemma-1},~\ref{L:key-lemma-2},~\ref{L:key-lemma-3} and
~\ref{L:around_xi}.  Those orders can be checked in the proofs of lemmas  and will not be emphasized later.
\end{rem}

The question of the singular limits of the Zakharov and related systems, the Klein-Gordon Zakharov system for example, has been studied extensively.
Quite often, the limiting solution (when it exists) satisfies a completely different nonlinear partial differential equations. The nonlinear Schr\"odinger limit of the Zakharov system is one physical problem involving plasma frequency and ion sound speed effects where such a singular limit process is interesting. The earlier results are shown in \cite{AA88, SW86} and the case when initial layer occurs was investigated by Ozawa and Tsutsumi in \cite{OT92a}. The readers are also referred to Masmoudi-Nakanishi \cite{MK2008} for a complete result where they were able to overcome the difficulty of the existence of a resonance frequency. The convergence of the quantum Zakharov system to the Zakharov system is also interesting, see \cite{GZG12} for the recent result. However, the convergence in the lower regularity developed in this paper is a challenging problem and it will be our main research project in the near future.

The rest of paper is organized as follows. In Section 2, we discuss the conservation
laws and hydrodynamic limit (formally) of the quantum Zakharov system. In Section 3, we reduce the well-posedness of Cauchy problem of the quantum Zakharov system to three key estimates. These estimates will be proved in Section 4 by using the estimates built up in Sections 5 and 6.
In Section 5, we build the estimate describing the interaction of two ``{\it fourth order waves}". Finally, we prove the estimate describing the interaction of
one ``{\it second order wave}" and one ``{\it fourth order wave}" in Section 6.

{\bf Notation} The expression $X\lesssim Y$ means that $X\leq CY$ for some constant $C$ depending on each occurrence.
The notation $ X\approx Y$ means that there exists two positive constants $C_1,C_2$ such that $C_1Y \leq X \leq C_2 Y$.
We also use the bracket $\lr{\xi}=(1+|\xi|)$ for the convenience. A constant $C(b)$ means the constant $C$ depends
on $b$, and a constant $C$ means that constant is a uniform constant.  The notation $B-$ appearing in section 4,5 and 6 means $B-\delta$ where $\delta$ is a small positive number
which can be chosen  arbitrarily close to zero.

We define the inner product in $L^2$ space by $\langle f(x),g(x)\rangle=\int f(x)\overline{g}(x)dx$. The Fourier transform
and inverse Fourier transform are defined by
\[
\wh{f}(\xi)=\int_{-\infty}^{\infty} e^{-ix\xi} f(x)dx\quad \hbox{and} \quad \check{f}(x)=\frac{1}{2\pi}\int_{-\infty}^{\infty} e^{ix\xi} \wh{f}(\xi)d\xi\;
\]
respectively. The Sobolev norm is defined by Fourier transform as
\[
\|f\|^2_{H^s}=\|\lr{\xi}^{s}\widehat{f}(\xi)\|^2_{L^2}=\int (1+|\xi|^2)^{s}|\wh{f}(\xi)|^2d\xi.
\]
When we consider the time-space Fourier transform and its inverse,
we use $(t,x)$ to denote the time-space variables and $(\tau,\xi)$ to denote their Fourier counters. The space-time Sobolve norm is defined by 
\[
\|u\|_{H^{s,b}}=\|\lr{\xi}^s\lr{\tau}^b\widehat{u}(\tau,\xi) \|_{L^2}.
\]

\section{Conservation law}

The one-dimensional quantum Zakharov equations (\ref{E:Zakharov-1})--(\ref{E:Zakharov-2})
are derived from a variational principle~\cite{Ha11},
\begin{equation}\label{action}
\delta S= \delta \int\!\!\!\int \mathcal{L} dx dt =0
\end{equation}
with a Lagrangian density
\begin{equation}
\begin{array}{c}\label{E:Lagrangian}
\displaystyle\mathcal{L}= \frac{i}{2}\Big(E^* \frac{\partial E}{\partial t}
-E \frac{\partial E^*}{\partial t}\Big)
-\Big|\frac{\partial E}{\partial x}\Big|^2-\frac{\partial u}{\partial x} |E|^2 +\frac{1}{2}\Big|\frac{\partial u}{\partial t}\Big|^2
-\frac{1}{2}\Big|\frac{\partial u}{\partial x}\Big|^2
\\ \\
\displaystyle-\varepsilon^2 \Big|\frac{\partial^2 E}{\partial x^2}\Big|^2
-\frac{\varepsilon^2}{2} \Big|\frac{\partial^2 u}{\partial x^2}\Big|^2,
\end{array}
\end{equation}
where the auxiliary variable $u$, satisfying $\frac{\partial u}{\partial x}=n$,  is introduced such that the density can be found. The variational derivative $\delta S/\delta E^*= \delta S/\delta E=0$ produces (\ref{E:Zakharov-1}) and its complex conjugate equation, respectively. But the equation for $n$ is not straightforward, instead, taking the variational derivative $\delta S/\delta u=0$ we have
\begin{equation}
\frac{\partial}{\partial x}\Big(|E|^2 +\frac{\partial u}{\partial x}\Big) +\frac{\partial^2 u}{\partial t^2}+\varepsilon^2\frac{\partial^4 u}{\partial x^4}=0
\end{equation}
which will reproduce (\ref{E:Zakharov-2}) after differentiation with respect to $x$. The Lagrangian formulation allows us to systematically derive conserved quantities by means of Noether's theorem, relating invariance, symmetries and conservation laws.
The action (\ref{action}) is trivially invariant under the phase transformation, i.e., gauge invariant,
$$
E(t,x) \mapsto e^{i\theta} E(t,x),\qquad \theta \in \mathbb{R}
$$
and thus quantum Zakharov equations admit the conservation law for the mass (or the number of high frequency quanta)
\begin{equation}\label{E:conservation-mass}
\int_{-\infty}^\infty|E|^2 dx=\int_{-\infty}^\infty E E^* dx\,.
\end{equation}
Similarly, the action is invariant under time translation, and we have the conservation for the energy (or Hamiltonian)
\begin{equation}\label{E:conservation-energy}
H=\int_{-\infty}^\infty\left( \Big|\frac{\partial E}{\partial x}\Big|^2+\varepsilon^2\Big|\frac{\partial^2 E}{\partial x^2}\Big|^2+n|E|^2+\frac{n^2}{2}+\frac{|V|^2}{2}
+\frac{\varepsilon^2}{2}\Big|\frac{\partial n}{\partial x}\Big|^2 \right) dx\,,
\end{equation}
where $V=-D^{-1}_{x} n_t$. From~\eqref{E:conservation-energy}, we
see that if $(E,n)\in H^2\oplus H^1$, one can control the energy $H$
since $|\int n|E|^2 dx|\leq \|n\|_{H^1}\|E\|_{H^2}$. Hence the local well-posedness result in Theorem~\ref{T:main-theorem}
implies the global well-posedness of the Cauchy problem~\eqref{E:Zakharov-1}-\eqref{E:Zakharov-2} in energy spaces
(cf.~\cite{GZG12}).

Besides the energy equation associated with the quantum Zakharov system (1.1) are the mass and momentum equations. Their densities $\rho$ and $J$  are given respectively by
\begin{equation}
\rho = E E^* = |E|^2,\qquad J= i\bigg(E \frac{\partial E^*}{\partial x}
- E^* \frac{\partial E}{\partial x}\bigg)
\end{equation}
and the mass and momentum equations are then
\begin{equation}\label{M-local}
\frac{\partial \rho}{\partial t} +\frac{\partial J}{\partial x}=\epsilon^2\frac{\partial J_Q}{\partial x} \,,
\end{equation}
\begin{equation}\label{M-locala}
\begin{array}{l}
\displaystyle{\frac{\partial J}{\partial t}}-
{\frac{\partial}{\partial x}}\bigg(E {\frac{\partial^2 E^*}{\partial x^2}}-2{\frac{\partial E}{\partial x}}{\frac{\partial E^*}{\partial x}} + E^* {\frac{\partial^2 E}{\partial x^2}} \bigg)+ 2 {\frac{\partial n}{\partial x}} |E|^2\\[4mm]
\displaystyle
=\varepsilon^2{\frac{\partial}{\partial x}}\bigg(E {\frac{\partial^4 E^*}{\partial x^4}}-2{\frac{\partial E}{\partial x}}
{\frac{\partial^3 E^*}{\partial x^3}} + 2{\frac{\partial^2 E}{\partial x^2}}{\frac{\partial^2 E^*}{\partial x^2}}
-2{\frac{\partial^3 E}{\partial x^3}}{\frac{\partial E^*}{\partial x}}-{\frac{\partial^4 E}{\partial x^4}} E^* \bigg),
\end{array}
\end{equation}
where
\begin{equation}
J_Q=i\bigg(E \frac{\partial^3 E^*}{\partial x^3} - \frac{\partial E}{\partial x}\frac{\partial^2 E^*}{\partial x^2}+ \frac{\partial^2 E}{\partial x^2}\frac{\partial E^*}{\partial x}-  \frac{\partial^3 E}{\partial x^3}E^*\bigg)\,.
\end{equation}
The conservation of mass  comes from the imaginary part of (\ref{E:Zakharov-1}) and hence contains no contribution of $n$.
The momentum $J$ is not conservative due to the coupling of the density $|E|^2$ and
${\frac{\partial n}{\partial x}}$. 
The $O(\epsilon^2)$ term of (\ref{M-local}) and (\ref{M-locala}) shows the quantum effect of the Langmuir wave. Formally letting $\varepsilon \to 0$ in (2.6) and (2.7) we have the hydrodynamical equations
\begin{equation}
\frac{\partial \rho}{\partial t} +\frac{\partial J}{\partial x}=0,
\end{equation}
\begin{equation}
\frac{\partial J}{\partial t} +\frac{\partial }{\partial x}\bigg(\frac{|J|^2}{\rho}\bigg)+2\rho \frac{\partial n}{\partial x}=
\frac{\partial}{\partial x}\bigg(\rho\frac{\partial^2}{\partial x^2} \log \rho\bigg),
\end{equation}
\begin{equation}
\frac{\partial^2 n}{\partial t^2}-\frac{\partial^2 n}{\partial x^2}=\frac{\partial^2 \rho}{\partial x^2},
\end{equation}
which are equivalent to the Zakharov equation (1.3a)-(1.3b) as long as the solutions are smooth.

\section{Reduction}
In order to solve equations~(\ref{E:Zakharov-1})-(\ref{E:Zakharov-3}),
we first split $n$ into positive and negative frequency parts according to
\[
n_{\pm}=n\pm i \Lambda^{-1}\partial_tn
\]
where $\Lambda=(-\Delta+\varepsilon^2\Delta^2)^{1/2}\equiv(-\Le)^{1/2}$.
Thus the quantum Zakharov
~\eqref{E:Zakharov-1},~\eqref{E:Zakharov-2},~\eqref{E:Zakharov-3} can be
transformed into the first order system
\begin{subequations}
\begin{align}
\label{E:qZakharov-2-1} & iE_t+\Le E=\frac{1}{2}(n_{+}+n_{-})E,\\
\label{E:qZakharov-2-2} & (i\partial_t\mp\Lambda)n_{\pm}=\mp\Lambda^{-1}(\Delta |E|^2),\\
\label{E:qZakharov-2-3} & E(0,x)=E_0(x),\quad n_{\pm}(0,x)=n_{0\pm}(x)=n_0(x)\pm i\Lambda^{-1}n_1(x).
\end{align}
\end{subequations}
First, we briefly review  Bourgain's method~\cite{Bo93}. The presentation
here is closely related to that of Ginibre, Tsutumi and Velo~\cite{GTV97}.
We want to solve the equation of the type
\begin{equation}\label{E:evolution-equation}
i\partial_t u-\Phi(-i\nabla)u=f(u)
\end{equation}
where $\Phi$ is a real function (or a real symmetric matrix valued function) defined in
$\mathbb{R}$ and $f$ is a nonlinear function. In our case, $u$ will be replaced by
$(E,n_{+},n_{-})$ and $\Phi(\xi)$ will be a diagonal matrix with entries
$(\Phie{\xi},\sqrt{\Phie{\xi}},-\sqrt{\Phie{\xi}})$
where $\Phie{\xi}=\xi^2+\varepsilon^2\xi^4$. The Cauchy problem
for~\eqref{E:evolution-equation} with  initial data $u(0)=u_0$ is rewritten as
the integral equation
\begin{equation}\label{E:evolution-equation-1}
\begin{split}
u(t)&=U(t)u_0-i\int_0^t U(t-s)f(u(s))ds \\
 &=U(t)u_0-iU*_{R}f(u)
\end{split}
\end{equation}
where $U(t)=\exp[-it\Phi(-i\nabla)]$ is the unitary group that solves the
linear equation and $*_{R}$ denotes the retarded convolution in time.
In order to solve the Cauchy problem locally in time for
some time interval $[-T,T]$, one introduces a time cut-off in~\eqref{E:evolution-equation-1}.
Let $\beta_1\in \mathit{C}^{\infty}_0$ be even, with $0\leq\beta_1\leq 1$, $\beta_1(t)=1$
for $-1\leq t\leq 1$ and $\beta_1(t)=0$ for $|t|\geq 2$ and let $\beta_{T}=\beta_{1}(t/T)$
for $0<T\leq 1$. Then one replaces the integral equation~\eqref{E:evolution-equation-1} by the cut-off equation
\[ 
u(t)=\beta_1(t)U(t)u_0-i\beta_{T}(t)\int_0^t U(t-s)f(\beta_{2T}(s)u(s))ds,
\]
since the solution of this equation is equal to the locally in time solution
of~~\eqref{E:evolution-equation-1}.

To solve the equation~\eqref{E:evolution-equation-1}, we define the Banach spaces $X=X^{s,b}$ 
as spaces of functions so that $U(-t)u$ belongs to the Sobolev space $H=H^{s,b}$, i.e.,
\begin{equation}\label{E:Bourgain-norm}
\normXs{u}:=\norm{U(-t)u}_{H^{s,b}}=\norm{\lr{\xi}^s\lr{\tau+\Phi(\xi)}^b 
\widehat{u}(\tau,\xi) }_{L^2}.
\end{equation}
By definition, we have
\[
\normXs{\beta_1U(t)u_0}=\norm{\beta_1u_0}_{H^{s,b}}=\norm{\beta_1(t)}_{H^b_t}\norm{u_0}_{H^s_x}.
\]
And we have the estimate for
$
\normXs{\beta_{T}(U*_{R}f(u))}.
$
\begin{lem}\label{L:GTV}\cite{GTV97}
Let $-1/2<b'\leq 0\leq b\leq b'+1$ and $T\leq 1$. Then
\[
\normXs{\beta_{T}(U*_{R}f(u))}\leq CT^{1-b+b'}\norm{f(u)}_{X^{s,b'}}.
\]
\end{lem}

From above lemma  we can build the local well-posedness of the Cauchy
problem in space $X^{s,b}$ by the standard contraction mapping argument if
$1-b+b'>0$ and non-linear estimates
of the form
\begin{equation}\label{E:non-linear}
\norm{f(u)}_{X^{s,b'}}\leq C\norm{u}^n_{X^{s,b}}
\end{equation}
hold for some power $n$ depending on $f$.
We also require $b>\frac{1}{2}$ to ensure
\[
X^{s,b}\subset \mathit{C}(\mathbb{R},H^s)
\]
and this completes the local well-posedness in $H^s$ spaces.
For more details, we refer the readers to~\cite{GTV97} or~\cite{JC11} and references therein.


Within this framework, we can solve~\eqref{E:qZakharov-2-1}--\eqref{E:qZakharov-2-3}.
Indeed, we define the operators
\begin{subequations}
\begin{align}
\label{E:Schrodinger}  \Se(t) E(x)=e^{it\Le}E(x)=\frac{1}{2\pi}\int e^{i(x\cdot\xi
-t\Phie{\xi})}\widehat{E}(\xi)d\xi, \\
\label{E:wave}  \We(t)n(x)=e^{\mp it\Lambda}n(x)=\frac{1}{2\pi}\int
e^{i(x\cdot\xi\pm t\sqrt{\Phie{\xi}} )}\widehat{n}(\xi) d\xi
\end{align}
\end{subequations}
where $\Phie{\xi}=\xi^2+\varepsilon^2\xi^4$.
The Cauchy problems for~(\ref{E:qZakharov-2-1})--(\ref{E:qZakharov-2-3}) can be solved by
showing the mapping $\Psi=(\Psi_0,\Psi_1)$ is a contracting mapping in a suitable Banach space
where $\Psi_0$ and $\Psi_1$ are defined respectively by
\begin{subequations}
\begin{align}
\label{E:Duhamel-1}  \Psi_0(E(t)) &=\Se(t)E_0-\frac{i}{2}\int_0^t \Se(t-s)(n_{+}(s)+n_{-}(s))E(s)ds, \\
\label{E:Duhamel-2}  \Psi_1(n_{\pm}(t)) &=\We(t)n_{0\pm}\mp i\int_0^t \We(t-s)\Lambda^{-1}(\Delta |E(s)|^2) ds.
\end{align}
\end{subequations}

Following~\eqref{E:Bourgain-norm}, the  norms $\normXk{E}$ and $\normXl{n}$ we need are given by
\[
\normXk{E}=\norm{\bracs{\xi}^k\bracs{\tau+\Phie{\xi}}^{b_1}\widehat{E}(\tau,\xi)}_{L^2_{\xi\tau}},
\]
\[
\normXl{n}=\norm{\bracs{\xi}^l\bracs{\tau\pm\sqrt{\Phie{\xi}}}^b\widehat{n}(\tau,\xi)}_{L^2_{\xi\tau}}.
\]
It is worth noting that the calculations of norms $\normXk{E}$ and $\normXl{n}$ can be done
by the duality argument.  Observe that $E\in X^{k,b_1}$ if and only if
$(1+|\xi|)^k(1+|\tau+\Phie{\xi}|)^{b_1}{\wh E}(\tau,\xi)\in L^2(\tau,\xi)$, that is
\[
\sup \langle {\wh E}(\tau,\xi) , d(\tau,\xi)\lr{\xi}^k\lr{\tau+\Phie{\xi}}^{b_1} \rangle
=\normXk{E}
\]
where the supremum is taken over all $d\in L^2(\tau,\xi)$  with $\|d\|_{L^2}\leq 1$.

For our purpose, we will use the Strichartz inequality for the fourth order
Schr\"odinger equation in our analysis. We define the non-homogeneous
differentiation operator
\begin{equation}\label{E:fractional-differentiation}
\begin{split}
D_{\varepsilon}^{\alpha}f(x)&=:\int_{\mathbb{R}}e^{ix\xi}
(1+6\varepsilon^2\xi^2)^{\frac{\alpha}{2}}{\wh f}(\xi) d\xi\\
&=\int_{\mathbb{R}}e^{ix\xi}
\lr{\xi_{\varepsilon}}^{{\alpha}}{\wh f}(\xi) d\xi.
\end{split}
\end{equation}
The notation $\lr{\xi_{\varepsilon}}^{{\alpha}}$ is called the symbol of
differentiation operator $D_{\varepsilon}^{\alpha}$.
The Strichartz inequality we need is the following.
\begin{lem} For $(t,x)\in\mathbb{R}^2$, we have
\[
\|D_{\varepsilon}^{1/2} \Se(t)E_0 \|_{L^4_t L^{\infty}_x}
 \leq C\|E_0\|_{L^2_x}
\]
\end{lem}
This lemma is a special case of Theorem 2.1 of~\cite{KPV91} where we take
$\phi(\xi)=\Phie{\xi}=\xi^2+\varepsilon^2\xi^4$ and $\theta=1$. The readers can
also see~\cite{BKS00,Pau07} for the Strichartz inequality of
the fourth order Schr\"odinger equation.

Similar to the Lemma 2.3 of~\cite{GTV97}, we have the following emdedding inequality
which is implied by Strichartz inequality.
\begin{lem}
 For any $b>1/2$, the inequality
\begin{equation}\label{E:Strichartz}
\|D_{\varepsilon}^{1/2} E(t,x) \|_{L^4_t L^{\infty}_x}
\leq C\|E(t,x)\|_{X^{0,b}}
\end{equation}
holds for all $E\in X^{0,b}$.
\end{lem}

Returning to our main stream of discussion for the local well-posedness of the Cauchy
problem, we have to prove the following.
\begin{thm}
The Cauchy problem~{\rm(\ref{E:qZakharov-2-1})-(\ref{E:qZakharov-2-3})} with initial data $(E_0,n_{0},n_{1})
\in H^{k}\bigoplus H^{l}\bigoplus H^{l-2}$ is locally well-posed in
$X^{k,b_1}\bigoplus X^{l,b}\bigoplus$ $X^{l-2,b}$ for any $(k,l)\in \mathbb{A}$ and 
$b>\frac{1}{2}\;,\;b_1>\frac{1}{2}$  close to $1/2$ but larger than $1/2$.
The set $\mathbb{A}$ is defined by
\begin{equation}\label{E:A-set}
\begin{split}
\mathbb{A}&=\left\{(k,l)|-\frac{3}{2}<k-l<\frac{3}{2}\;,\;k\geq 0 \right\}\\
& \cup \left\{(k,l)|2k-l>-\frac{3}{2}\;,\;k+l>-\frac{3}{2}\;,\; -\frac{3}{4} <k<0\right\}.
\end{split}
\end{equation}
\end{thm}
\begin{proof}
Here we sketch the proof.
The discussion before~\eqref{E:Schrodinger} and~\eqref{E:wave} indicates that we need the corresponding estimates
of~\eqref{E:non-linear} in our setting, i.e.,
\[
\|n_{\pm}E\|_{X^{k,b_1'}}\leq C(\varepsilon)\normXl{n_{\pm}}\normXk{E}
\]
and
\[
\|\Lambda^{-1}(\Delta |E|^2)\|_{X^{l,b'}_{\pm}}\leq C(\varepsilon)\normXk{E}^2.
\]
These nonlinear estimates will be given by Lemma~\ref{L:bilinear-1} and
Lemma~\ref{L:bilinear-2}.
\end{proof}

\begin{lem}\label{L:bilinear-1} For any $(k,l)$ which is inside the union of $\{(k,l)|k-l< \frac{3}{2},k\geq 0\}$
and  $\{(k,l)|k+l>-\frac{3}{2}, k<0\}$, we have
\begin{equation}\label{E:bilinear-1}
\|n_{\pm}E\|_{X^{k,b_1'}}\leq C(\varepsilon)\normXl{n_{\pm}}\normXk{E}
\end{equation}
by choosing $b>1/2,b_1>1/2$ close to $1/2$,  $b'<-1/2,b_1'>-1/2$ 
close to $-1/2$ (see Remark~\ref{R:3.6}  below) and
$-1/2<b_1'\leq 0\leq b_1\leq 1+b_1'$, $-1/2<b'\leq 0\leq b\leq 1+b'$.
\end{lem}
\begin{rem}\label{R:3.6}
In order to get the optimal region of $(k,l)$ in the above lemma as well as in the below lemma, we can for example take $b=b_1=\frac{1}{2}+\frac{1}{4}\theta$ and $b'=b_1'=-\frac{1}{2}+\theta$ where $\theta$
is a small positive number arbitrarily close to zero. This choice also gives $1-b+b'>0,\;1-b_1+b_1'>0$ which is need in Lemma~\ref{L:GTV} to ensure the power of 
$T$ is positive.
However we keep notations $b,b',b_1,b_1'$ in these two lemmas in order to track
the exponents easily.
\end{rem}

\begin{proof}
Let $-c_1=b_1'$.
To estimate $\norm{n_{\pm}E}_{X^{k,-c_1}}$, we take its scalar product with a generic function $E_1$ in $X^{-k,c_1}$, i.e.
\begin{equation}\label{E:bilinear-1-dual}
\begin{split}
&\int {\overline{\wh E}}_1(\tau_1,\xi_1)\widehat{n}_{\pm}\ast\widehat{E}(\tau_1,\xi_1) d\tau_1 d\xi_1 \\
&=\iint {\overline{\wh E}}_1(\tau_1,\xi_1) \wh{n}_{\pm}(\tau,\xi) \wh{E}(\tau_2,\xi_2) d\tau_1 d\xi_1d\tau_2 d\xi_2
\end{split}
\end{equation}
where
\begin{equation}\label{E:conservation-convolution}
\xi=\xi_1-\xi_2 \;,\quad  \tau=\tau_1-\tau_2.
\end{equation}

Let $\wh{v}_1=\lr{\xi_1}^{-k}\lr{\tau_1+\Phieo}^{c_1}{\overline{\wh E}}_1(\tau_1,\xi_1)=:\lr{\xi_1}^{-k}\lr{\sigma_1}^{c_1}{\overline{\wh E}}_1(\tau_1,\xi_1)$, then $\|\widehat{v}_1\|_{L^2}=\|E_1\|_{X^{-k,c_1}}$.  Similarly we define $\widehat{v_2}=\lr{\xi_2}^k\lr{\sigma_2}^{b_1}\widehat{E}(\tau_2,\xi_2)$ and
$\widehat{v}=\lr{\xi}^l\lr{\sigma}^b\widehat{n_{\pm}}(\tau,\xi)$
where $\sigma_2=\tau_2+\Phiew,\sigma=\tau\pm
\sqrt{\Phie{\xi}}$ so that $\|\wh{v_2}\|_{L^2}=\normXk{E}$ and
$\norm{\wh v}_{L^2}=\normXl{n_{\pm}}$ .

We define~\eqref{E:bilinear-1-dual} as
\[ 
\begin{split}
S&:=\iint
 \widehat{v}_1(\tau_1,\xi_1)\frac{\lr{\xi_1}^k}{\lr{\tau_1+\Phieo}^{c_1}}\widehat{v}(\tau,\xi)\frac{1}{\lr{\xi}^l\lr{\tau\pm\sqrt{\Phie{\xi}}}^{b}} \\
&{\hspace{4 cm}}\widehat{v}_2(\tau_2,\xi_2)\frac{1}{\lr{\xi_2}^k \lr{\tau_2+\Phie{\xi_2}}^{b_1}} d\tau_1 d\xi_1d\tau_2 d\xi_2 \\
&=\iint\frac{\hv\hv_1\hv_2\lr{\xi_1}^k}
{\lr{\sigma}^b\lr{\sigma_1}^{c_1}\lr{\sigma_2}^{b_1}\lr{\xi_2}^k\lr{\xi}^l}d\tau_1 d\xi_1d\tau_2 d\xi_2.
\end{split}
\]
Thus to prove~\eqref{E:bilinear-1} is equivalent to showing that
\begin{equation}\label{E:bound-S}
|S|\leq C(\varepsilon) \norm{v}_{L^2}\norm{v_1}_{L^2}\norm{v_2}_{L^2}.
\end{equation}
We shall prove~\eqref{E:bound-S} by discussing two different cases, i.e.
$|\xi|\leq 1$ and $|\xi|>1$.

\noindent{\bf Case 1.} $|\xi|\leq 1.$ \par
We note that $|\xi|\leq 1$ implies
\[
C_1(l,k)\leq \frac{\lr{\xi_1}^k}{\lr{\xi_2}^k\lr{\xi}^l}\leq C_2(l,k).
\]
Hence we only have to show that
\begin{equation}\label{E:bound-S_1}
|S_1|\leq C(\varepsilon) \norm{v}_{L^2}\norm{v_1}_{L^2}\norm{v_2}_{L^2}
\end{equation}
where
\[
S_1=\iint\frac{|\hv\hv_1\hv_2|}
{\lr{\sigma}^b\lr{\sigma_1}^{c_1}\lr{\sigma_2}^{b_1}}d\tau_1 d\xi_1d\tau_2 d\xi_2\,.
\]
The proof of~\eqref{E:bound-S_1} follows the same idea as in Lemma 3.2 of~\cite{GTV97}
with modification.
The slight difference here is that the Strichartz inequality of the fourth order Schr\"odinger
equation has more regularity than that of the Schr\"odinger equation. We shall not need 
this fact for the proof of~\eqref{E:bound-S_1} but it will be used in the next lemma.
Hence we write down the detail of the proof for the convenience of readers and for the later use.

Since $S_1$ is decreasing as a function of $b,b_1$ and $c_1$, it is sufficient to prove the case where
$b=\frac{1}{4}$ and $b_1=c_1=\frac{1}{4}+$. Then
\[
\begin{split}
S_1&=\int{\Big (}\frac{|\hv_1|}{\lr{\sigma_1}^{1/4+}}{\Big )}
{\Big (}\int\frac{|\hv\hv_2|}{\lr{\sigma}^{1/4}\lr{\sigma_2}^{1/4+}}
d\tau_2d\xi_2 {\Big )} d\tau_1d\xi_1\\
&=\int{\Big (}\frac{|\hv_1|}{\lr{\sigma_1}^{1/4+}}{\Big )}^{\vee}
{\Big (}\int\frac{|\hv\hv_2|}{\lr{\sigma}^{1/4}\lr{\sigma_2}^{1/4+}}
d\tau_2 d\xi_2 {\Big )}^{\vee} dt_1 dx_1\\
&=\int{\Big (}\frac{|\hv_1|}{\lr{\sigma_1}^{1/4+}}{\Big )}^{\vee}
{\Big (}\frac{|\hv|}{\lr{\sigma}^{1/4}} {\Big )}^{\vee}
{\Big (}\frac{|\hv_2|}{\lr{\sigma_2}^{1/4+}} {\Big )}^{\vee}dt_1 dx_1 \\
&\leq {\Big \|}{\Big (}\frac{|\hv_1|}{\lr{\sigma_1}^{1/4+}}{\Big )^{\vee}}
{\Big \|}_{L^{8/3}_tL^4_x}
 {\Big \|}{\Big (}\frac{|\hv|}{\lr{\sigma}^{1/4}}{\Big )^{\vee}}
{\Big \|}_{L^{4}_tL^2_x}
{\Big \|}{\Big (}\frac{|\hv_2|}{\lr{\sigma_2}^{1/4+}}{\Big )^{\vee}}
{\Big \|}_{L^{8/3}_tL^4_x}.
\end{split}
\]
Here
\[
\begin{split}
{\Big \|}{\Big (}\frac{|\hv|}{\lr{\sigma}^{1/4}}{\Big )^{\vee}}
{\Big \|}_{L^{4}_tL^2_x}
&= {\Big \|}\int e^{it\tau}\frac{|\wh{v}(\tau,\xi)|}
{\lr{\tau\pm\sqrt{\Phie{\xi}}}^{1/4}}d\tau{\Big \|}_{L^{4}_tL^2_{\xi}}\\
&\leq {\Big \|}\int e^{it\tau}\frac{|\wh{v}(\tau,\xi)|}
{\lr{\tau\pm\sqrt{\Phie{\xi}}}^{1/4}}d\tau{\Big \|}_{L^2_{\xi}L^{4}_t}
\;({\rm Minkowski\;inequality})\\
&\leq {\Big \|}\int e^{it\tau}\frac{|{\wh v}(\tau\mp\sqrt{\Phie{\xi}},\xi)|}{\lr{\tau}^{1/4}} d\tau {\Big \|}_{L^2_{\xi}L^{4}_t}\\
&=\|\lr{\partial_t}^{-1/4}(|{\wh v}( t,\xi)|e^{\pm it\sqrt{\Phie{\xi}}})\|_{L^2_{\xi}L^{4}_t}\\
&\leq C\||{\wh v}(\xi, t)|e^{\pm it\sqrt{\Phie{\xi}}}\|_{L^2_{\xi}L^2_t}
\quad ({\rm Sobolev\; imbedding})\\
&=C\|v\|_{L^2_{\xi}L^2_t}.
\end{split}
\]
We interpolate the equality $\||{\wh v}_j|^{\vee}\|_{L^2_tL^2_x}=\|v_j\|_{L^2_{\tau}L^2_{\xi}}\;,\;j=1,2$, ~\eqref{E:Strichartz} and the Strichartz-type inequality
 $\|(\lr{\sigma_j}^{-b}|{\wh v}_j|)^{\vee}\|_{L^4_tL^{\infty}_x}\leq
C\|D_{\varepsilon}^{-1/2} v_j\|_{L^2_{\tau}L^2_{\xi}}\;,\;b>1/2$, to obtain
\begin{equation}\label{E:Strichartz-inter}
\bigg\|\Big(\frac{|{\wh v}_j|}{\lr{\sigma_j}^{1/4+}}\Big)^{\vee}\bigg\|_{L^{8/3}_tL^{4}_x}\leq
C\|D_{\varepsilon}^{-1/4}v_j\|_{L^2_{\tau}L^2_{\xi}}
\leq C\| v_j\|_{L^2_{\tau}L^2_{\xi}}.
\end{equation}
Therefore~\eqref{E:bound-S_1} is proved.

\noindent{\bf Case 2.} $|\xi|>1$.\par
For simplicity of representation,
let $\zeta=(\tau,\xi)$ and $\zeta_i=(\tau_i,\xi_i)$ and note $\zeta=\zeta_1-\zeta_2$
from~\eqref{E:conservation-convolution}. We also let $K(\zeta_1,\zeta_2)$ equal
\[
\frac{\lr{\xi_1}^k}{\lr{\sigma}^b\lr{\sigma_1}^{c_1}\lr{\sigma_2}^{b_1}
\lr{\xi_2}^k\lr{\xi}^l}.
\]
Now the integral $S$ defined before Case 1. is of the form
\[
S=\iint \wh{v}(\zeta)\wh{v}_1(\zeta_1)\wh{v}_2(\zeta_2)K(\zeta_1,\zeta_2) d\zeta_1d\zeta_2.
\]

By the Schwarz inequality we have
\begin{equation}\label{E:Schwartz_inequality}
\begin{split}
|S|^2&\leq \norm{v}^2_{L^2}\int {\Big |}\int \wh{v}_1(\zeta_1)\wh{v}_2(\zeta_1-\zeta)K(\zeta_1,\zeta_1-\zeta)  d\zeta_1{\Big |}^2 d\zeta\\
&\leq \norm{v}^2_{L^2}{\Big\{}\sup\limits_{\zeta}\int|K(\zeta_1,\zeta_1-\zeta)|^2d\zeta_1 {\Big\}} \iint |\wh{v}_1(\zeta_1)\wh{v}_2(\zeta_1-\zeta)|^2
 d\zeta d\zeta_1\\
&= C_1^2(\varepsilon) \norm{v}^2_{L^2}\norm{v_1}^2_{L^2}\norm{v_2}^2_{L^2}
\end{split}
\end{equation}
where
\begin{equation}\label{E:C-1}
\begin{split}
 C_1(\varepsilon)&=\sup\limits_{\zeta,|\xi|\geq 1}\bigg(\int|K(\zeta_1,\zeta_1-\zeta)|^2d\zeta_1\bigg)^{1/2}\\
  &=\sup\limits_{\tau,|\xi|\geq 1}\bigg{\{}\frac{1}{\lr{\xi}^l\lr{\tau\pm\sqrt{\Phie{\xi}}}^{b}}\times \\
&\hspace{1cm}{\Big (}\int\frac{\lr{\xi_{1}}^{2k}d\tau_1d\xi_1}{\lr{\tau_1+\Phie{\xi_1}}^{2c_1}\lr{\tau_1-\tau
+\Phie{\xi_1-\xi}}^{2b_1}\lr{\xi_1-\xi}^{2k}}{\Big )}^{1/2}\bigg{\}}.
\end{split}
\end{equation}
Hence, the proof is completed if we can show that $C_1(\varepsilon)$ is bounded 
which will be proved in Lemma~\ref{L:C_1_bounded}.
\end{proof}

Another nonlinear estimate we need is the following.

\begin{lem}\label{L:bilinear-2}
If $(k,l)$ lies inside the union of the domains $\{(k,l)|k-l> -\frac{3}{2},k\geq 0\}$
and $\{(k,l)|2k-l>-\frac{3}{2}, -3/4< k<0\}$, then we have
\begin{equation}\label{E:bilinear-2}
\|\Lambda^{-1}(\Delta |E|^2)\|_{X^{l,b'}_{\pm}}\leq C(\varepsilon)\normXk{E}^2
\end{equation}
by choosing $b>1/2,b_1>1/2$ close to $1/2$, $b'<-1/2,b_1'>-1/2$ close to $-1/2$ such that $-1/2<b_1'\leq 0\leq b_1\leq 1+b_1'$ and $-1/2<b'\leq 0\leq b\leq 1+b'$
(see Remark~\ref{R:3.6}).
\end{lem}
\begin{proof}

Let $-c=b'$.
To estimate $\|\Lambda^{-1}(\Delta |E|^2)\|_{X^{l,-c}_{\pm}}$,
we take its scalar product with a generic function $n\in X^{-l,c}$,
 i.e.
\begin{equation}\label{E:bilinear-2-dual}
\begin{split}
&\int {\overline{\wh n}}(\tau,\xi)(\Phie{\xi})^{-1/2}|\xi|^2\wh{\overline{E}}\ast\widehat{E}(\tau,\xi) d\tau d\xi \\
&=\iint {\overline{\wh n}}(\tau,\xi)(\Phie{\xi})^{-1/2}|\xi|^2
{\overline{\wh E}}\big(-(\tau-\tau_1),-(\xi-\xi_1)\big) \wh{E}(\tau_1,\xi_1) d\tau_1 d\xi_1d\tau d\xi\\
&=\iint {\overline{\wh n}}(\tau,\xi)(\Phie{\xi})^{-1/2}|\xi|^2
{\overline{\wh E}}(\tau_2,\xi_2) \wh{E}(\tau_1,\xi_1) d\tau_1 d\xi_1d\tau d\xi
\end{split}
\end{equation}
where
\begin{equation}
\xi_2=\xi_1-\xi \;,\; \tau_2=\tau_1-\tau.
\end{equation}
Let $\widehat{v}=\lr{\xi}^{-l}\lr{\sigma}^c{\overline{\wh n}}(\tau,\xi)$ where
$\sigma=\tau\pm\sqrt{\Phie{\xi}}$ then $\norm{\wh{v}}_{L^2}
=\norm{n}_{X_{\pm}^{-l,c}}$.  Similarly we define $\wh{v}_1=\lr{\xi_1}^{k}\lr{\sigma_1}^{b_1}\wh{E}(\tau_1,\xi_1)$ and $\wh{v_2}=\lr{\xi_2}^k\lr{\sigma_2}^{b_1}{\overline{\wh E}}(\tau_2,\xi_2)$ where $\sigma_i=\tau_i+\Phie{\xi_i}$,
then $\|\widehat{v}_i\|_{L^2}=\normXk{E_i}$.
To proceed, we rewrite~\eqref{E:bilinear-2-dual} as
\begin{equation}\label{E:intgeral-W-pre}
\iint\frac{\hv\hv_1\hv_2|\xi|^2\lr{\xi}^l}
{\lr{\sigma}^c\lr{\sigma_1}^{b_1}\lr{\sigma_2}^{b_1}(\Phie{\xi})^{1/2}\lr{\xi_1}^k\lr{\xi_2}^k}d\tau_1 d\xi_1d\tau d\xi.
\end{equation}
We note that the estimate remains unchanged if we put an absolute value sign to the integrand. We also note that
\[
\frac{|\xi|^2}{(\Phie{\xi})^{1/2}}\leq \frac{1}{\varepsilon}.
\]
Thus to estimate~\eqref{E:intgeral-W-pre} is equal to estimating
\begin{equation}\label{E:intgeral-W}
W=\iint\frac{\hv\hv_1\hv_2\lr{\xi}^l}
{\lr{\sigma}^c\lr{\sigma_1}^{b_1}\lr{\sigma_2}^{b_1}
\lr{\xi_1}^k\lr{\xi_2}^k}d\tau_1 d\xi_1d\tau d\xi.
\end{equation}
Therefore to prove~\eqref{E:bilinear-2} is equivalent to showing that
\begin{equation}\label{E:bound-W}
|W|\leq C(\varepsilon)\norm{v}_{L^2}\norm{v_1}_{L^2}\norm{v_2}_{L^2}.
\end{equation}
In order to prove~\eqref{E:bound-W}, we should discuss two cases, i.e.
$k>-\frac{1}{4}$ and $-\frac{3}{4}< k\leq -\frac{1}{4}$.

{\bf Case 1.} $-\frac{1}{4}<k $.
There are two sub-cases; $|\xi|\leq 1$ and $|\xi|>1$.

\noindent{\bf Case 1a.} $|\xi|\leq 1$. \par
The argument here is similar to Case 1 in the previous lemma. It is clear that we only have to show
\begin{equation}\label{E:bound-W-1}
|W_1|\leq C(\varepsilon)\norm{v}_{L^2}\norm{v_1}_{L^2}\norm{v_2}_{L^2}
\end{equation}
where
\begin{equation}\label{E:intgeral-W-1}
W_1=\iint\frac{|\hv\hv_1\hv_2|\lr{\xi_1}^{1/4}\lr{\xi_2}^{1/4}}
{\lr{\sigma}^{1/4}\lr{\sigma_1}^{1/4+}\lr{\sigma_2}^{1/4+}
}d\tau_1 d\xi_1d\tau d\xi.
\end{equation}
This can be done by the same method as in the Case 1 of Lemma~\ref{L:bilinear-1}.
Here we have
\[
W_1\leq {\Big \|}{\Big (}\frac{\lr{\xi_1}^{1/4}|\hv_1|}{\lr{\sigma_1}^{1/4+}}{\Big )^{\vee}}
{\Big \|}_{L^{8/3}_tL^4_x}
 {\Big \|}{\Big (}\frac{|\hv|}{\lr{\sigma}^{1/4}}{\Big )^{\vee}}
{\Big \|}_{L^{4}_tL^2_x}
{\Big \|}{\Big (}\frac{\lr{\xi_2}^{1/4}|\hv_2|}{\lr{\sigma_2}^{1/4+}}{\Big )^{\vee}}
{\Big \|}_{L^{8/3}_tL^4_x}.
\]
The terms $\lr{\xi_i}^{1/4}$ in the first and third norms do not hurt. Since
$\lr{\xi_i}\leq C(\varepsilon)\lr{(\xi_i)_\varepsilon}$ (recall~\eqref{E:fractional-differentiation}
for notation $\lr{\xi_\varepsilon}$)  means~\eqref{E:Strichartz} or~\eqref{E:Strichartz-inter} is equal to
\begin{equation}
\Big\|\Big(\frac{|\lr{\xi_j}^{1/4}{\wh v}_j|}{\lr{\sigma_j}^{1/4+}}\Big)^{\vee}\Big\|_{L^{8/3}_tL^{4}_x}\leq
C(\varepsilon)\|v_j\|_{L^2_{\tau}L^2_{\xi}}.
\end{equation}

\noindent{\bf Case 1b.} $|\xi|>1$. \par
By applying argument as in~\eqref{E:Schwartz_inequality}, we see that
 the proof of~\eqref{E:bound-W}  is reduced to the boundedness of
\begin{equation}\label{E:C-2}
\begin{split}
&C_2(\varepsilon)=\sup\limits_{\tau,|\xi|>1}\bigg{\{}\frac{\lr{\xi}^l}{\lr{\tau\pm\sqrt{\Phie{\xi}}}^{c}}\times  \\
&\hspace{1.5cm}{\Big (}\int\frac{\lr{\xi_1}^{-2k}
\lr{\xi_1-\xi}^{-2k}d\tau_1d\xi_1}{\lr{\tau_1-\tau
+\Phie{\xi_1-\xi}}^{2b_1}\lr{\tau_1+\Phie{\xi_1}}^{2b_1}}{\Big )}^{1/2}\bigg{\}}.
\end{split}
\end{equation}
The boundedness of $C_2(\varepsilon)$ is proved in Lemma~\ref{L:C_2_bounded}..

{\bf Case 2.} $-\frac{3}{4}<k\leq -\frac{1}{4}$.

To complete the proof, we have to prove that~\eqref{E:bound-W} holds for
the missing part $\{(k,l)|2k-l>-3/2,
-\frac{3}{4}<k\leq -\frac{1}{4}\}$ (see Figure \ref{F:klm}).

For this purpose, we consider the following.
Integrating with respect to $\zeta_1,\zeta_2$ instead of $\zeta_1,\zeta$ in ~\eqref{E:intgeral-W} (Jacobian is 1) and applying a similar argument as in~\eqref{E:Schwartz_inequality} reduces the proof of~\eqref{E:bound-W} in this case
to proving the boundedness of
\[
\begin{split}
&C_3(\varepsilon)=\sup\limits_{\xi_2,\tau_2}\bigg{\{}\frac{1}{\lr{\tau_2
+\Phie{\xi_2}}^{b_1}\lr{\xi_2}^{k}}\times \\
&\hspace{1.5cm}{\Big (}\int
\frac{\lr{\xi_1-\xi_2}^{2l} d\tau_1d\xi_1}{\lr{\tau_1-\tau_2\pm\sqrt{\Phie{\xi_1-\xi_2}}}^{2c}\lr{\tau_1+\Phie{\xi_1}}^{2b_1}
\lr{\xi_{1}}^{2k}}{\Big )}^{1/2}\bigg{\}}.
\end{split}
\]

For the convenience of further discussions, we relabel the variables and write them as
\begin{equation}\label{E:C-3}
\begin{split}
&C_3(\varepsilon)=\sup\limits_{\xi,\tau}\bigg{\{}\frac{1}{\lr{\tau
+\Phie{\xi}}^{b_1}\lr{\xi}^{k}}\times \\
&\hspace{1.5cm}{\Big (}\int
\frac{\lr{\xi_1-\xi}^{2l} d\tau_1d\xi_1}{\lr{\tau_1-\tau\pm\sqrt{\Phie{\xi_1-\xi}}}^{-2b'}\lr{\tau_1+\Phie{\xi_1}}^{2b_1}
\lr{\xi_{1}}^{2k}
}{\Big )}^{1/2}\bigg{\}}.
\end{split}
\end{equation}
In Lemma~\ref{L:C_3_bounded}, we should prove that $C_3(\varepsilon)$
is bounded in the desired $(k,l)$ region.

\end{proof}

\begin{rem}
Note that condition $k\leq -1/4$ is required in order to prove $C_3$ is bounded.
(See the proof of Case 1c. of Lemma~\ref{L:key-lemma-3}.)
Thus we have to discuss two cases when $k<0$.
\end{rem}

\section{Proof of Theorem}

In order to complete the proof of Theorem~\ref{T:main-theorem}, we have to show that
$C_1,C_2$ and $C_3$ are bounded. In fact, we will prove $C_1$ is bounded in the following first lemma.

\begin{lem}\label{L:C_1_bounded}
Let $(k,l)$ be in the union of the set $\{(k,l)|k-l< \frac{3}{2},k\geq 0\}$ and
$\{(k,l)|k+l>-\frac{3}{2}, k<0\}$. There exists a suitable $\theta>0$
 close enough to $0$, such that if we let
$b=b_1=\frac{1}{2}+\frac{1}{2}\theta$ and $b_1'=-\frac{1}{2}+\theta$, then
\begin{equation}\label{E:key-estimate}
\begin{split}
&\sup\limits_{|\xi|\geq1,\tau}\bigg{\{}\frac{1}{\lr{\xi}^l\lr{\tau\pm\sqrt{\Phie{\xi}}}^{b}}\times \\
&\hspace{1cm}{\Big (}\int_{-\infty}^{\infty}\int_{-\infty}^{\infty}\frac{\lr{\xi_1}^{2k}d\tau_1d\xi_1}
{\lr{\tau_1+\Phie{\xi_1}}^{-2b_1'}
\lr{\tau-\tau_1+\Phie{\xi_1-\xi}}^{2b_1}\lr{\xi-\xi_1}^{2k}}{\Big)^{1/2}}\bigg{\}}\\
&\leq C(\varepsilon,b,b_1).
\end{split}
\end{equation}
\end{lem}
\begin{proof}
Let $B_1=-b_1'$ and $B_2=b_1$. By
Lemma~\ref{L:key-lemma-1}, we have
\[
\begin{split}
&{\Big (}\int_{-\infty}^{\infty}\int^{\infty}_{-\infty}\frac{\lr{\xi_1}^{2k}d\tau_1d\xi_1}{\lr{\tau_1+\Phie{\xi_1}}^{-2b_1'}
\lr{\tau-\tau_1+\Phie{\xi_1-\xi}}^{2b_1}\lr{\xi-\xi_1}^{2k}}{\Big )^{1/2}}\\
&\leq C(B,\varepsilon)\frac{\lr{\xi}^{|k|}}{\lr{\Gamma}^{(B-\frac{1}{8})}}
\end{split}
\]
where $\lr{\Gamma}=\lr{\xi^2+\varepsilon^2\xi^4}$ and $B=-b_1'$.
Therefore, the problem is reduced to proving
\[
\sup\limits_{\xi,\tau}{\bigg\{ }\frac{1}{\lr{\xi}^l\lr{\tau\pm\sqrt{\Phie{\xi}}}^{b}}\times \frac{\lr{\xi}^{|k|}}
{\lr{\Gamma}^{(B-\frac{1}{8})}}{\bigg\} }\leq C(b,b_1,\varepsilon).
\]
It is easy to see that for any $\tau$ we have
\begin{equation}
C(B,\varepsilon)\lr{\xi}^{\frac{3}{2}-}\leq(1+|\tau\pm\sqrt{\xi^2+\varepsilon^2\xi^4}|)^b(1+|\xi^2+\varepsilon^2\xi^4|)^{(B-\frac{1}{8})}
\end{equation}
where $\frac{3}{2}-=\frac{3}{2}-\delta$ and $\delta$ is a small positive number depending on $\theta$.
Therefore we conclude that~\eqref{E:key-estimate} holds when
\[
\frac{3}{2}-\geq |k|-l.
\]
\end{proof}

In the following lemma we will prove that $C_2(\varepsilon)$ is uniformly
bounded with respect to $\tau$ and $\xi$.

\begin{lem}\label{L:C_2_bounded}
Given $(k,l)$ belonging to the union of the set $\{(k,l)|k-l> -\frac{3}{2},k\geq 0\}$ and $\{(k,l)|2k-l>-\frac{3}{2}, -1/4< k <0\}$, there is a suitable $\theta>0$
close enough to $0$ such that if we let
$b_1=\frac{1}{2}+\frac{1}{2}\theta$ and $b'=-\frac{1}{2}+\theta$, then
\begin{equation}\label{E:key-estimate-2}
\begin{split}
&\sup\limits_{|\xi|\geq 1,\tau}\bigg{\{}\frac{\lr{\xi}^l}{\lr{\tau\pm\sqrt{\Phie{\xi}}}^{-b'}}\times \\
&\hspace{1cm}{\Big (}\int_{-\infty}^{\infty}\int_{-\infty}^{\infty}\frac{\lr{\xi_1}^{-2k}\lr{\xi_1-\xi}^{-2k}
d\tau_1d\xi_1}{\lr{\tau_1+\Phie{\xi_1}}^{2b_1}\lr{\tau-\tau_1
+\Phie{\xi_1-\xi}}^{2b_1}}{\Big )}^{1/2}\bigg{\}}\\
&\leq C(\varepsilon,b_1,b').
\end{split}
\end{equation}
\end{lem}
\begin{proof} In order to apply Lemma~\ref{L:key-lemma-2} to the estimate of~\eqref{E:key-estimate-2}, we
replace $\lr{\tau_1+\Phie{\xi_1}}^{2b_1}$ in the left hand
side of~\eqref{E:key-estimate-2} by $\lr{\tau_1+\Phie{\xi_1}}^{-2b'}$.
It is clear that this does not affect the result.
Let $B_1=-b'$ and $B_2=b_1$, then we have $B=-b'$. 
Now we discuss two cases as the following.

{\bf Case 1.} $k\geq 0$. \par
We are reduced to proving that
\[
\sup\limits_{\xi,\tau}{\bigg\{ }\frac{\lr{\xi}^l}{\lr{\tau\pm\sqrt{\Phie{\xi}}}^{-b'}}\times
\frac{\lr{\xi}^{-k}}{\lr{\Gamma}^{(B-\frac{1}{8})}}{\bigg\} }\leq C(b',b_1,\varepsilon)
\]
where $\lr{\Gamma}=\lr{\xi^2+\varepsilon^2\xi^4}$. From
\[
C(B,\varepsilon)\lr{\xi}^{\frac{3}{2}-}\leq(1+|\tau\pm\sqrt{\xi^2+\varepsilon^2\xi^4}|)^{-b'}
(1+|\xi^2+\varepsilon^2\xi^4|)^{(B-\frac{1}{8})}
\]
we see that it holds when $l-k\leq\frac{3}{2}-$.

{\bf Case 2.} $-1/4< k<0$.  \par
\[
\sup\limits_{\xi,\tau}{\bigg\{ }\frac{\lr{\xi}^l}{\lr{\tau\pm\sqrt{\Phie{\xi}}}^{-b'}}\times
\frac{\max\{\lr{A}^{-2k},\lr{\xi}^{-2k}\}}{\lr{\Gamma}^{(B-\frac{1}{8})}}{\bigg\} }
\leq C(b',b_1,\varepsilon)
\]
where $\Gamma=\xi^2+\varepsilon^2\xi^4$ and $A=\varepsilon^{-2/3}(|\frac{\tau}{\xi}|)^{1/3}$.
When $\lr{\xi}>\lr{A}$, the bound holds for
$l-2k\leq\frac{3}{2}-$. On the other hand, when $\lr{\xi}\leq\lr{A}$ we have $|\tau|>\xi^4$. That is
$\lr{\tau\pm\sqrt{\Phie{\xi}}}^{-b'}\approx\lr{\tau}^{-b'}$ which is larger than
$\lr{\tau}^{-\frac{2k}{3}}$ when $-1/4<k<0$.

\end{proof}

Finally, we prove that $C_3$ is bounded.
\begin{lem}\label{L:C_3_bounded}
If $(k,l)$ is in  the set  $\{(k,l)|2k-l>-3/2,
-\frac{3}{4}<k\leq -\frac{1}{4}\}$, we can find a suitable $\theta>0$
closing enough to $0$, such that if we let
$b_1=\frac{1}{2}+\frac{1}{2}\theta$ and $b'=-\frac{1}{2}+\theta$, then

\begin{equation}\label{E:key-estimate-3}
\begin{split}
&\sup\limits_{\xi,\tau}\bigg{\{}\frac{1}{\lr{\xi}^k\lr{\tau+\Phie{\xi}}^{b_1}}\times \\
&\hspace{1cm}{\Big (}\int_{-\infty}^{\infty}\int_{-\infty}^{\infty}\frac{\lr{\xi_1-\xi}^{2l}d\tau_1d\xi_1}
{\lr{\tau_1+\Phie{\xi_1}}^{2b_1}\lr{\tau_1-\tau\pm\sqrt{\Phie{\xi_1-\xi}}}^{-2b'}\lr{\xi_1}^{2k}}{\Big )}^{1/2}\bigg{\}}\\
&\leq C(\varepsilon ,b_1,b').
\end{split}
\end{equation}
\end{lem}

\begin{proof}
Let $B_1=-b'=\frac{1}{2}-\theta,
\;B_2=b_1=\frac{1}{2}+\frac{1}{2}\theta$ and apply Lemma~\ref{L:key-lemma-3},
 then we have the following two cases for discussing with $B=-b'$.

{\bf Case 1.} $|\xi|>32\varepsilon^{-2}$. \par
\noindent{\bf Case 1a.} $k\leq-1/4\;,\;0\leq l$ and $8B-2(l-k)>1$.\par
We have
\begin{equation}\label{E:C-3-case-1}
\begin{split}
&\eqref{E:key-estimate-3}\leq C(B,k,l,\varepsilon)\sup\limits_{\xi,\tau}\bigg{\{}
{\Big (}\frac{1}{\lr{\xi}^k\lr{\tau+\Phie{\xi}}^{b_1}}\Big{)}\\
&{\hskip 3cm}\times
\Big{(}F(\xi,\Lambda_{+})+F(\xi,\Lambda_{-}){\Big )}\bigg{\}}
\end{split}
\end{equation}
where
\[
F(\xi,\Lambda_+)=\frac{
\max\{\lr{\xi},\lr{\Lambda_+}^{1/4}\}^{l-k}}{\lr{\Lambda_+}^{(B-\frac{1}{8})-}}
\;,\quad
F(\xi,\Lambda_-)=\frac{
\max\{\lr{\xi},\lr{\Lambda_-}^{1/4}\}^{l-k}}{\lr{\Lambda_-}^{(B-\frac{1}{8})-}}.
\]
Note that $\lr{\Lambda_+}=\lr{1}$ if $-\sqrt{\Phie{\xi}}\leq\tau\leq -\frac{1}{2}\sqrt{\Phie{\xi}}$ and $\lr{\Lambda_-}=\lr{1}$ if
$\sqrt{\Phie{\xi}}\leq\tau\leq \frac{3}{2}\sqrt{\Phie{\xi}}$,
otherwise $\Lambda_+=\tau+c\sqrt{\Phie{\xi}}$ for some positive $c$ and
$\Lambda_-=\tau+c\sqrt{\Phie{\xi}}$ for some negative $c$ where $3/4\leq |c|\leq 5/4$.
We note that if $|\tau|$ is larger than $2\Phie{\xi}$, the~\eqref{E:C-3-case-1}
is bounded provided $2k-l\geq -7/2$ by comparing the exponent of $\tau$ and noting 
that $\lr{\xi}\leq C(\varepsilon)|\tau|^{1/4}$. 
If $|\tau|$ is smaller than $2\Phie{\xi}$, the sizes of $\Lambda+,\Lambda-$ are 
not larger than $\Phie{\xi}$. On the other hand, the products 
$\lr{\tau+\Phie{\xi}}\lr{\Lambda_+}$ and $\lr{\tau+\Phie{\xi}}\lr{\Lambda_-}$ 
is large than $C(\varepsilon)\lr{\xi}^4$. Thus we conclude that~\eqref{E:C-3-case-1}
is bounded provided $2k-l\geq-(\frac{3}{2}-)$ in this case.  
 Finally we note that $2k-l\geq-(\frac{3}{2}-)$ is 
more strict than $2k-l\geq -7/2$. 

\noindent{\bf Case 1b.} $k\leq-1/4\;,\; l<0$ and $8B-2(l-k)>1$.\par
We have to consider the additional term
\[
\frac{1}{\lr{\xi}^k\lr{\tau+\Phie{\xi}}^{b_1}}\times \frac{\lr{\xi}^{-k}}{\lr{\xi}^{3B-}}
\]
which is bounded if $-\frac{3}{4}<k$. (When $B$ is close to $1/2$.)

{\bf Case 2.} $|\xi|\leq 32\varepsilon^{-2}$.\par

When $|\xi|\leq 32\varepsilon^{-2}$ and $8B-2(l-k)>1$, we have the bound
\begin{equation}
\eqref{E:key-estimate-3}\leq C(B,k,l,\varepsilon)\sup\limits_{\xi,\tau}\bigg{\{}
\frac{1}{\lr{\xi}^k\lr{\tau+\Phie{\xi}}^{b_1}}\bigg{\}}\leq C(B,k,l,\varepsilon).\\
\end{equation}
Note that the condition $8B-2(l-k)>1$ is equal to $k-l>-3/2$ when $B$ is close to $1/2$.

\end{proof}

\section{4th order wave vs 4th order wave}


In this section, we shall discuss two lemmas which describe the non-linear interactions
between two fourth order waves, i.e. $\lr{\tau_1+\Phie{\xi_1}}$ and
$\lr{\tau-\tau_1+\Phie{\xi_1-\xi}}$, when they are away from each other ($|\xi|>1$).
For these two models, we are able to characterize the interactions of two waves 
quantitatively. Lemma~\ref{L:key-lemma-1} characterizes one type of the interaction
of two fourth order waves which is needed in the nonlinear estimate of ``{\it Schr\"odinger part}". On the other hand,
Lemma~\ref{L:key-lemma-2} describes another type of  interaction between
two fourth order waves which is needed in the nonlinear estimate of ``{\it Wave part}". In both lemmas
we can see that the  majority of interaction is again a fourth order wave
represented by $\Gamma=\xi^2+\varepsilon^2\xi^4$.

We first introduce an elementary estimate.
\begin{lem}\label{L:s-1_s-2}
Let $0\leq a_{-}\leq a_{+}$ and $a_{+}+a_{-}>1/2$. The following estimate holds
for  $s_1,s_2\in\mathbb{R}$
\[
\int^{\infty}_{-\infty}\frac{dx}{(1+|x-s_1|)^{2a_{-}}(1+|x-s_2|)^{2a_{+}}}\leq C
\frac{1}{(1+|s_1-s_2|)^{\alpha}}
\]
where $\alpha=2a_{-}-[1-2a_{+}]_{+}$.
\end{lem}
The notation  $[\lambda]_{+}=\lambda$ if $\lambda>0\;,\;=\epsilon>0$ if $\lambda=0$ and $=0$ if $\lambda<0$.
This lemma is a variant of Lemma 4.2 in~\cite{GTV97}. We can simply
consider the change of variable $y=x-\frac{s_1+s_2}{2}$, take $s=\frac{s_1-s_2}{2}$ in
Lemma 4.2 of~\cite{GTV97} and get above lemma.

\begin{lem}[Two 4th order waves]\label{L:key-lemma-1}
Assume $1/4< B_1\leq B_2$ and $1/6<B<1/2$ where $2B=2B_1-[1-2B_2]_{+}$.
When $|\xi|> 1$ there exists a constant $C(B,\varepsilon)$ which is independent of
$\tau$ and $\xi$ such that
\begin{equation}\label{E:key-estimate-xi-1}
\begin{split}
&{\bigg (}\int_{-\infty}^{\infty}\int^{\infty}_{-\infty}\frac{\lr{\xi_1}^{2k}d\tau_1d\xi_1}{\lr{\tau_1+\Phie{\xi_1}}^{2B_1}
\lr{\tau_1-\tau+\Phie{\xi_1-\xi}}^{2B_2}\lr{\xi_1-\xi}^{2k}}{\bigg )^{1/2}}\\
&\leq C(B,\varepsilon)\frac{\lr{\xi}^{|k|}}{\lr{\Gamma}^{(B-\frac{1}{8})}}
\end{split}
\end{equation}
where $\Gamma=\xi^2+\varepsilon^2\xi^4$.
\end{lem}
\begin{proof}

The condition $|\xi|>1$ implies $|\xi|\approx\lr{\xi}$ and $\Gamma\approx\lr{\Gamma}$. We can switch
between them in the following discussion
by multiplying a uniform constant. \par

By Lemma~\ref{L:s-1_s-2}, we have
\[
\int_{-\infty}^{\infty}\frac{d\tau_1}{\lr{\tau_1+\Phie{\xi_1}}^{2B_1}
\lr{\tau_1-\tau+\Phie{\xi_1-\xi}}^{2B_2}}\leq \frac{C}{\lr{\tau-\Phie{\xi_1-\xi}+\Phie{\xi_1}}^{2B}}.
\]

Hence we are reduced to proving that
\[
{\bigg(}\int_{-\infty}^{\infty}\frac{\lr{\xi_1}^{2k}d\xi_1}{
\lr{\tau-\Phie{\xi_1-\xi}+\Phie{\xi_1}}^{2B}\lr{\xi_1-\xi}^{2k}}{\bigg)^{1/2}}\leq
 C(B,\varepsilon)\frac{
\lr{\xi}^{|k|}}{\lr{\Gamma}^{(B-\frac{1}{8})}}.
\]

Let $\xi_1=\frac{1}{2}\xi+\eta$. We note that
\[
\begin{split}
\tau-\Phie{\xi_1-\xi}+\Phie{\xi_1}&=\tau-\Phie{\eta-\frac{1}{2}\xi}+\Phie{\eta+\frac{1}{2}\xi}\\
&=\tau+(2\xi+\varepsilon^2\xi^3)\eta+4\varepsilon^2\xi\eta^3
\end{split}
\]
and the integral becomes
\[
\begin{split}
&\int_{-\infty}^{\infty}\frac{\lr{\xi_1}^{2k}}{\lr{\tau
-\Phie{\xi_1-\xi}+\Phie{\xi_1}}^{2B}\lr{\xi-\xi_1}^{2k}}\,d\xi_1\\[2mm]
&\quad =\int_{-\infty}^{\infty}
\frac{\lr{\eta+\frac{1}{2}\xi}^{2k}}{\lr{\tau+(2\xi+\varepsilon^2\xi^3)\eta+4\varepsilon^2\xi\eta^3}^{2B}
\lr{\eta-\frac{1}{2}\xi}^{2k}}\,d\eta.
\end{split}
\]
The key observation is that the following inequality
\begin{equation}\label{E:key_observation}
\lr{\xi}^{-2|k|}\leq \frac{\lr{\xi_1}^{2k}}{\lr{\xi_1-\xi}^{2k}}
=\frac{\lr{\eta+\frac{1}{2}\xi}^{2k}}{\lr{\eta-\frac{1}{2}\xi}^{2k}} \leq \lr{\xi}^{2|k|}
\end{equation}
holds for $k\in\mathbb{R}$. From which we have
\[
\begin{split}
&\int_{-\infty}^{\infty}
\frac{\lr{\eta+\frac{1}{2}\xi}^{2k}}{\lr{\tau+(2\xi+\varepsilon^2\xi^3)\eta+4\varepsilon^2\xi\eta^3}^{2B}
\lr{\eta-\frac{1}{2}\xi}^{2k}}\,d\eta\\[2mm]
&\qquad \leq
\lr{\xi}^{2|k|}\int_{-\infty}^{\infty}
\frac{d\eta}{\lr{\tau+(2\xi+\varepsilon^2\xi^3)\eta+4\varepsilon^2\xi\eta^3}^{2B}}.
\end{split}
\]
Hence we are reduced to proving that
\begin{equation}\label{E:key-integration-1-0}
\int_{-\infty}^{\infty}\frac{d\eta}{\lr{\tau+(2\xi+\varepsilon^2\xi^3)\eta+4\varepsilon^2\xi\eta^3}^{2B}}
\leq C(B,\varepsilon)\frac{1}{\lr{\Gamma}^{(2B-\frac{1}{4})}}.
\end{equation}
We rewrite the left hand side as
\begin{equation}\label{E:key-integration-1}
\begin{split}
&\int_{0}^{\infty}\frac{d\eta}{(1+|\tau-(2\xi+\varepsilon^2\xi^3)\eta-4\varepsilon^2\xi\eta^3|)^{2B}}\\[2mm]
&{\hskip 1.5cm} +\int_{0}^{\infty}\frac{d\eta}{(1+|\tau+(2\xi+\varepsilon^2\xi^3)\eta+4\varepsilon^2\xi\eta^3|)^{2B}}.
\end{split}
\end{equation}
From this representation, we see that it suffices to study the case when $\xi$ is positive, that is $\xi>1$.
In the following we should discuss different cases according to the sign of $\tau$.
We note that it is sufficient to estimate the second integral due to that 
~\eqref{E:key-integration-1} is symmetric with respect to $\tau$.
For simplicity of notation, we let $f_{\tau,\xi}(\eta)=\tau+(2\xi+\varepsilon^2\xi^3)\eta+4\varepsilon^2\xi\eta^3$.

\noindent{\bf Case 1.} $\tau\geq 0$. Let $\eta=\varepsilon^{-1/2} \lr{\Gamma}^{1/4}\eta'$.\par
Note that $\lr{\Gamma}^{1/4}\leq 2\xi $ when $\xi\geq 1$. Then we can estimate the integral as
\begin{equation}\label{E:key-integral-1-1a}
\begin{split}
&\int_0^{\xi}\frac{d\eta}{(1+|f_{\tau,\xi}(\eta)|)^{2B}}+\int_{\xi}^{\infty}\frac{d\eta}{(1+|f_{\tau,\xi}(\eta)|)^{2B}}\\
&\leq\int_0^{\xi}\frac{d\eta}{(1+(\xi+\varepsilon^2\xi^3)\eta)^{2B}}
+\int_{\xi}^{\infty}\frac{d\eta}{(1+\Gamma+\varepsilon^2\xi\eta^3)^{2B}}\\
&\leq\frac{1}{(1-2B)}\frac{1}{(\xi+\varepsilon^2\xi^3)}(1+\Gamma)^{1-2B}\\
&{\hspace{2cm}}+\frac{\varepsilon^{-1/2}\lr{\Gamma}^{1/4}}{\lr{\Gamma}^{2B}}
\int_{\frac{\varepsilon^{1/2}\xi}{\lr{\Gamma}^{1/4}}}^{\infty}
\left(1+\frac{\varepsilon^{1/2}\xi}{\lr{\Gamma}^{1/4}}\eta'^3\right)^{-2B}d\eta'.\\
\end{split}
\end{equation}
The first term is bounded by
\[
 C(B,\varepsilon)\frac{1}{\lr{\Gamma}^{2B-\frac{1}{4}}}
\]
by using $\Gamma=\xi(\xi+\varepsilon^2\xi^3)$. The second term in the last line of~\eqref{E:key-integral-1-1a}
enjoys the same bound provided $1/6<B<1/2$ since the coefficient of $\eta'^3$ is larger than $2^{-1}\varepsilon^{1/2}$.

\noindent{\bf Case 2.} $\tau<0$.\par

Let $R$ be the root of $f_{\tau,\xi}(\eta)=0$.  We consider the following decomposition
\begin{equation}\label{E:key-integral-1-1b}
\begin{split}
\int_0^R\frac{d\eta}{(1+|f_{\tau,\xi}(\eta)|)^{2B}}+\int_R^{R+\xi}\frac{d\eta}{(1+|f_{\tau,\xi}(\eta)|)^{2B}}+
\int_{R+\xi}^{\infty}\frac{d\eta}{(1+|f_{\tau,\xi}(\eta)|)^{2B}}.
\end{split}
\end{equation}
We begin with the estimate of the second integral. Observe that
\[
1+|f_{\tau,\xi}(R)+(\xi+\varepsilon^2\xi^3)r|<1+|f_{\tau,\xi}(R+r)|
\]
holds for $0\leq r\leq \xi$. Therefore
\begin{equation}\label{E:A-to-A-plus-xi}
\int_R^{R+\xi}\frac{d\eta}{(1+|f_{\tau,\xi}(\eta)|)^{2B}}\leq\int_0^{\xi}\frac{d\eta}{(1+(\xi+\varepsilon^2\xi^3)\eta)^{2B}}
\leq C(B,\varepsilon)\frac{1}{\lr{\Gamma}^{2B-\frac{1}{4}}}.
\end{equation}
Similarly for $s\geq 0$, we have
\[
\begin{split}
1+\Gamma+\varepsilon^2\xi(\xi+s)^3&=1+|f_{\tau,\xi}(R)+(\xi+\varepsilon^2\xi^3)\xi+\varepsilon^2\xi(\xi+s)^3|\\
&<1+|f_{\tau,\xi}(R+\xi+s)|.
\end{split}
\]
Hence
\[
\int_{R+\xi}^{\infty}\frac{d\eta}{(1+|f_{\tau,\xi}(\eta)|)^{2B}}\leq
\int_{\xi}^{\infty}\frac{d\eta}{(1+\Gamma+\varepsilon^2\xi\eta^3)^{2B}}
\leq C(B,\varepsilon)\frac{1}{\lr{\Gamma}^{2B-\frac{1}{4}}}.
\]
Now we turn to the estimate of the first integral of~\eqref{E:key-integral-1-1b}.
Let $P$ be the number such that $PR=-\tau$. It is easy to see that
\begin{equation}\label{E:P}
P>(\xi+\varepsilon^2\xi^3).
\end{equation}
By convexity of $f_{\tau,\xi}(\eta)$, we have
\[
(1+|f_{\tau,\xi}(\eta)|)\geq (1+|\tau+P\eta|)
\]
for $0\leq\eta\leq R$.  Thus we have
\[
\begin{split}
\int_0^R\frac{d\eta}{(1+|f_{\tau,\xi}(\eta)|)^{2B}}&\leq\int_0^{R}\frac{d\eta}{(1-\tau-P\eta)^{2B}}\\
&=\frac{1}{(1-2B)}\frac{1}{-P}(1-(1-\tau)^{1-2B})\\
&\leq C(B)\frac{(-\tau)^{1-2B}}{P}.
\end{split}
\]
The last inequality follows from the fact that $0<1-2B<1$.

There are two sub-cases to discuss.
In case $|\tau|\leq \Gamma$, we combine~\eqref{E:P} and above to get
\begin{equation}\label{E:estimate_in_1_1b}
\frac{(-\tau)^{1-2B}}{P}\leq\frac{\lr{\Gamma}^{1-2B}}{(\xi+\varepsilon^2\xi^3)}\cdot\frac{\xi}{\xi}
\leq\frac{2\xi}{\lr{\Gamma}^{2B}}\leq C(B,\varepsilon)\frac{1}{\lr{\Gamma}^{2B-\frac{1}{4}}}.
\end{equation}
In case $|\tau|>|\Gamma|$, we use the fact
\[
R\leq\varepsilon^{-2/3}(\frac{-\tau}{\xi})^{1/3}.
\]
Combining above inequalities and the relation $PR=-\tau$, we have
\[
\frac{(-\tau)^{1-2B}}{P}=\frac{R}{(-\tau)^{2B}}\leq \frac{C(B,\varepsilon)}{\lr{\tau}^{2B-1/3}\lr{\xi}^{1/3}}
\leq  \frac{C(B,\varepsilon)}{\lr{\Gamma}^{2B-1/3}\lr{\xi}^{1/3}}\leq \frac{C(B,\varepsilon)}{\lr{\Gamma}^{2B-1/4}}
\]
since $B>1/6$.
\end{proof}

\begin{lem}[Two 4th order waves]\label{L:key-lemma-2}
Assume $1/4< B_1\leq B_2$ and $1/3<B<1/2$ where $2B=2B_1-[1-2B_2]_{+}$.
When $|\xi|> 1$ there exists a constant $C$ independent of $\tau$ and $\xi$ such that
for $k\geq 0$
\begin{equation}\label{E:key-estimate-xi-2}
\begin{split}
&{\bigg (}\int_{-\infty}^{\infty}\int_{-\infty}^{\infty}\frac{\lr{\xi_1}^{-2k}
\lr{\xi_1-\xi}^{-2k}}
{\lr{\tau_1+\Phie{\xi_1}}^{2B_1}\lr{\tau_1-\tau+\Phie{\xi_1-\xi}}^{2B_2}}d\tau_1 d\xi_1{\bigg )^{1/2}}\\
&\qquad \leq C(B,\varepsilon)\frac{
\lr{\xi}^{-k}}{\lr{\Gamma}^{(B-\frac{1}{8})}}
\end{split}
\end{equation}
and for $-1/4\leq k<0$
\begin{equation}\label{E:key-estimate-xi-3}
\begin{split}
&{\bigg (}\int_{-\infty}^{\infty}\int_{-\infty}^{\infty}\frac{\lr{\xi_1}^{-2k}
\lr{\xi_1-\xi}^{-2k}}
{\lr{\tau_1+\Phie{\xi_1}}^{2B_1}\lr{\tau_1-\tau+\Phie{\xi_1-\xi}}^{2B_2}}d\tau_1 d\xi_1{\bigg )^{1/2}}\\
&\qquad \leq C(B,k,\varepsilon)\frac{
\max\{\lr{A}^{-2k},\lr{\xi}^{-2k}\}}{\lr{\Gamma}^{(B-\frac{1}{8})}}
\end{split}
\end{equation}
where $\Gamma=\xi^2+\varepsilon^2\xi^4$ and $A=\varepsilon^{-2/3}(|\frac{\tau}{\xi}|)^{1/3}$.
\end{lem}
\begin{proof}

By Lemma~\ref{L:s-1_s-2}, we have
\[
\int_{-\infty}^{\infty}\frac{d\tau_1}{\lr{\tau_1+\Phie{\xi_1}}^{2B_1}\lr{\tau_1-\tau+
\Phie{\xi_1-\xi}}^{2B_2}}\leq \frac{C}{\lr{\tau-\Phie{\xi_1-\xi}+\Phie{\xi_1}}^{2B}}.
\]
Hence the problem is reduced to proving that
\begin{equation}
{\bigg (}\int_{-\infty}^{\infty}\frac{\lr{\xi_1}^{-2k}\lr{\xi_1-\xi}^{-2k}}{\lr{\tau
-\Phie{\xi_1-\xi}+\Phie{\xi_1}}^{2B}}\,d\xi_1{\bigg )^{1/2}}\leq C(\varepsilon)\frac{
\lr{\xi}^{-k}}{\lr{\Gamma}^{(B-\frac{1}{8})}}
\end{equation}
for $k\geq 0$ and
\begin{equation}
{\bigg (}\int_{-\infty}^{\infty}\frac{\lr{\xi_1}^{-2k}\lr{\xi_1-\xi}^{-2k}}{\lr{\tau
-\Phie{\xi_1-\xi}+\Phie{\xi_1}}^{2B}}\, d\xi_1{\bigg )^{1/2}}\leq C(\varepsilon)\frac{
\max\{\lr{A}^{-2k},\lr{\xi}^{-2k}\}}{\lr{\Gamma}^{(B-\frac{1}{8})}}
\end{equation}
for $k<0$. Let $\xi_1=\frac{1}{2}\xi+\eta$. We note that
\[
\begin{split}
\tau-\Phie{\xi_1-\xi}+\Phie{\xi_1}
&=\tau-\Phie{\eta-\frac{1}{2}\xi}+\Phie{\eta+\frac{1}{2}\xi}\\
&=\tau+(2\xi+\varepsilon^2\xi^3)\eta+4\varepsilon^2\xi\eta^3.
\end{split}
\]
Then
\[
\begin{split}
&\int_{-\infty}^{\infty}\frac{\lr{\xi_1}^{-2k}\lr{\xi-\xi_1}^{-2k}}{\lr{\tau
-\Phie{\xi_1-\xi}+\Phie{\xi_1}}^{2B}}\, d\xi_1\\[2mm]
&\qquad =\int_{-\infty}^{\infty}
\frac{\lr{\eta+\frac{1}{2}\xi}^{-2k}\lr{\eta-\frac{1}{2}\xi}^{-2k}}{\lr{\tau+(2\xi+\varepsilon^2\xi^3)\eta+4\varepsilon^2\xi\eta^3}^{2B}
}\,d\eta.
\end{split}
\]

{\bf Case 1.} $k\geq 0$. \par
Employing the inequality
\[
\lr{\xi}^{2k}\leq C(k)\lr{\eta-\frac{1}{2}\xi}^{2k}\lr{\eta+\frac{1}{2}\xi}^{2k}
\]
and the inequality~\eqref{E:key-integration-1-0} in the Lemma~\ref{L:key-lemma-1}, we can conclude~\eqref{E:key-estimate-xi-2}.

{\bf Case 2.} $-1/4\leq k<0$.\par
We note that $\lr{\eta-\frac{1}{2}\xi}^{2k}\lr{\eta+\frac{1}{2}\xi}^{2k}$ is an even function of
$\eta$. Let $f_{\tau,\xi}(\eta)=\tau+2(\xi+\varepsilon^2\xi^3)\eta+4\varepsilon^2\xi\eta^3$.
By the same discussion as in~\eqref{E:key-integration-1} in the proof of Lemma~\ref{L:key-lemma-1}, we only have to show
\[
\int_0^{\infty}\frac{\lr{\eta-\frac{1}{2}\xi}^{-2k}\lr{\eta+\frac{1}{2}\xi}^{-2k}}
{(1+|f_{\tau,\xi}(\eta)|)^{2B}}\,d\eta
\leq C(\varepsilon)\frac{
\max\{\lr{A}^{-4k},\lr{\xi}^{-4k}\}}{\lr{\Gamma}^{(2B-\frac{1}{4})}}
\]
for $\xi>1$. Using the fact $$
\lr{\eta-\frac{1}{2}\xi}^{-2k}\lr{\eta+\frac{1}{2}\xi}^{-2k}\leq \lr{\eta+\xi}^{-4k},\quad \eta>0,\xi>0,
$$
it suffices to show
\begin{equation}
\int_0^{\infty}\frac{\lr{\eta+\xi}^{-4k}}
{(1+|f_{\tau,\xi}(\eta)|)^{2B}}\,d\eta
\leq C(\varepsilon)\frac{
\max\{\lr{A}^{-4k},\lr{\xi}^{-4k}\}}{\lr{\Gamma}^{(2B-\frac{1}{4})}}.
\end{equation}
There are two sub-cases to consider in the following.

\noindent{\bf Case 2a.} $\tau\geq 0$.\par
Because $\lr{\Gamma}^{1/4}\leq 2\xi $ when $\xi>1$, we can estimate the integral as
\begin{equation}
\begin{split}
&\int_0^{\xi}\frac{\lr{\eta+\xi}^{-4k}}{(1+|f_{\tau,\xi}(\eta)|)^{2B}}\,d\eta+
\int_{\xi}^{\infty}\frac{\lr{\eta+\xi}^{-4k}}{(1+|f_{\tau,\xi}(\eta)|)^{2B}}\,d\eta\\
&\leq 2\lr{\xi}^{-4k}\int_0^{\xi}\frac{d\eta}{(1+(\xi+\varepsilon^2\xi^3)\eta)^{2B}}
+\int_{\xi}^{\infty}\frac{\lr{\eta+\xi}^{-4k}}{(1+\Gamma+\varepsilon^2\xi\eta^3)^{2B}}\,d\eta.\\
\end{split}
\end{equation}
From estimate~\eqref{E:key-integral-1-1a}, the first term of the last line is bounded by
\[
 C(B,\varepsilon)\frac{\lr{\xi}^{-4k}}{\lr{\Gamma}^{2B-\frac{1}{4}}}.
\]
Let $\eta=\varepsilon^{-1/2} \lr{\Gamma}^{1/4}\eta'$. We can estimate the last integral of (5.15) as
\begin{equation}\label{E:xi-to-infty}
\begin{split}
\int_{\xi}^{\infty}&\frac{\lr{\eta+\xi}^{-4k}}
{(1+\Gamma+\varepsilon^2\xi\eta^3)^{2B}}\,d\eta\\
&=\lr{\xi}^{-4k}\frac{1}{\lr{\Gamma}^{2B}}
\int_{\xi}^{\infty}\left(1+\frac{\eta}{\lr{\xi}}\right)^{-4k}\left(1+\frac{\varepsilon^2\xi\eta^3}{\lr{\Gamma}^{1/4}}\right)^{-2B}d\eta\\
&=\lr{\xi}^{-4k}\frac{\varepsilon^{-1/2}\lr{\Gamma}^{1/4}}{\lr{\Gamma}^{2B}}
   \int_{\frac{\varepsilon^{2/3}\xi}{\lr{\Gamma}^{1/4}}}^{\infty}
\left(1+\frac{\varepsilon^{-1/2} \lr{\Gamma}^{1/4}\eta'}{\lr{\xi}}\right)^{-4k}\\
&\hskip4cm
\left(1+\frac{\varepsilon^{1/2}\xi}{\lr{\Gamma}^{1/4}}\eta'^3\right)^{-2B}d\eta' \\
&\leq  C(B,\varepsilon)\frac{\lr{\xi}^{-4k}}{\lr{\Gamma}^{2B-\frac{1}{4}}}\int_{\frac{\varepsilon^{2/3}\xi}{\lr{\Gamma}^{1/4}}}
^{\infty}(1+2\eta')^{-4k}
(1+\varepsilon^{1/2}\eta'^3)^{-2B}d\eta'\\
&\leq  C(B,k,\varepsilon)\frac{\lr{\xi}^{-4k}}{\lr{\Gamma}^{2B-\frac{1}{4}}}
\end{split}
\end{equation}
provided $6B+4k>1$. (Note that the condition $-1/4\leq k<0$ implies $1/3<B$.)

\noindent{\bf Case 2b.} $\tau<0$. Let $R$ be the root of $f_{\tau,\xi}(\eta)=0$.
It is easy to see
\begin{equation}\label{E:R}
R\leq \varepsilon^{-2/3}\left(\frac{-\tau}{\xi}\right)^{1/3}=A.
\end{equation}
We consider the following decomposition
\begin{equation}\label{E:split-2-case-2}
\begin{split}
\int_0^R\frac{\lr{\eta+\xi}^{-4k}d\eta}{(1+|f_{\tau,\xi}(\eta)|)^{2B}}+\int_R^{R+\xi}\frac{\lr{\eta+\xi}^{-4k}d\eta}{(1+|f_{\tau,\xi}(\eta)|)^{2B}}+
\int_{R+\xi}^{\infty}\frac{\lr{\eta+\xi}^{-4k}d\eta}{(1+|f_{\tau,\xi}(\eta)|)^{2B}}.
\end{split}
\end{equation}
The numerator of the integrand in the first two integrals in~\eqref{E:split-2-case-2} is clearly bounded by  $\max\{R,\xi\}^{-4k}$. From the estimates~\eqref{E:key-integral-1-1b} and~\eqref{E:A-to-A-plus-xi} we know these two integrals are bounded by
\[
C(B,\varepsilon)\frac{\max\{\lr{R},\lr{\xi}\}^{-4k}}{\lr{\Gamma}^{2B-\frac{1}{4}}}
\leq C(B,\varepsilon)\frac{\max\{\lr{A},\lr{\xi}\}^{-4k}}{\lr{\Gamma}^{2B-\frac{1}{4}}}.
\]
Next, we note that  for $s\geq 0$
\[
\begin{split}
1+\Gamma+\varepsilon^2\xi(\xi+s)^3&=1+|f_{\tau,\xi}(R)+(\xi+\varepsilon^2\xi^3)\xi+\varepsilon^2\xi(\xi+s)^3|\\
&<1+|f_{\tau,\xi}(R+\xi+s)|.
\end{split}
\]
Hence, for the third integral of~\eqref{E:split-2-case-2} we have
\[
\int_{R+\xi}^{\infty}\frac{\lr{\eta+\xi}^{-4k}}{(1+|f_{\tau,\xi}(\eta)|)^{2B}}\,d\eta
\leq
\int_{\xi}^{\infty}\frac{\lr{\eta+(R+\xi)}^{-4k}}
{(1+\Gamma+\varepsilon^2\xi\eta^3)^{2B}}\,d\eta.
\]
Similar to~\eqref{E:xi-to-infty}, we have
\[
\begin{split}
&\int_{\xi}^{\infty}\frac{\lr{\eta+(R+\xi)}^{-4k}}
{(1+\Gamma+\varepsilon^2\xi\eta^3)^{2B}}\,d\eta\\
&=\lr{R+\xi}^{-4k}\frac{1}{\lr{\Gamma}^{2B}}
\int_{\xi}^{\infty}\left(1+\frac{\eta}{\lr{R+\xi}}\right)^{-4k}\left(1+\frac{\varepsilon^2\xi\eta^3}{\lr{\Gamma}^{1/4}}\right)^{-2B}d\eta\\
&\leq  C(B,\varepsilon)\frac{\max\{\lr{R},\lr{\xi}\}^{-4k}}{\lr{\Gamma}^{2B-\frac{1}{4}}}
\leq C(B,\varepsilon)\frac{\max\{\lr{A},\lr{\xi}\}^{-4k}}{\lr{\Gamma}^{2B-\frac{1}{4}}}.
\end{split}
\]

\end{proof}

\section{2nd order wave vs 4th order wave}

In this section we shall prove Lemma~\ref{L:key-lemma-3} which characterizes the interaction of
a second order wave $\lr{\tau_1-\tau\pm\sqrt{\Phie{\xi_1-\xi}}}$ with a fourth order wave
$\lr{\tau_1+\Phie{\xi_1}}$. The majority part of
interaction possesses the features of the second order wave as well as those of the
fourth order wave. Comparing to Lemma~\ref{L:key-lemma-1} and~\ref{L:key-lemma-2}, we
see that its magnitude is of second order while the power of wave is determined by the
fourth order wave.

\begin{lem}\label{L:key-lemma-3}
Assume $0\leq B_1\leq B_2$, $B_1+B_2>1/2$. Let $2B=2B_1-[1-2B_2]_{+}$. We also assume
$1/8<B<1/2$.
When $|\xi|>32\varepsilon^{-2}$, $k\leq-1/4\;,\;l\geq 0$ and $8B-2(l-k)>1$
there exists a constant $C(B,k,l,\varepsilon)$ which is independent of $\tau,\xi$ such that
\begin{equation}\label{E:key-estimate-xi-4}
\begin{split}
&{\bigg(}\int_{-\infty}^{\infty}\int_{-\infty}^{\infty}
\frac{\lr{\xi_1-\xi}^{2l}\lr{\xi_1}^{-2k}d\tau_1d\xi_1}{\lr{\tau_1+\Phie{\xi_1}}^{2B_2}
\lr{\tau_1-\tau\pm\sqrt{\Phie{\xi_1-\xi}}}^{2B_1}}{\bigg)^{1/2}}\\
&\leq C(B,k,l,\varepsilon){\Big\{}F(\xi,\Lambda_{+})+F(\xi,\Lambda_{-})
{\Big\}}
\end{split}
\end{equation}
where
\[
F(\xi,\Lambda_+)=\frac{
\max\{\lr{\xi},\lr{\Lambda_+}^{1/4}\}^{l-k}}{\lr{\Lambda_+}^{(B-\frac{1}{8})-}}
\;,\;
F(\xi,\Lambda_-)=\frac{
\max\{\lr{\xi},\lr{\Lambda_-}^{1/4}\}^{l-k}}{\lr{\Lambda_-}^{(B-\frac{1}{8})-}}.
\]
The number $\lr{\Lambda_+}=\lr{1}$ if $-\sqrt{\Phie{\xi}}\leq\tau\leq -\frac{1}{2}\sqrt{\Phie{\xi}}$ and $\lr{\Lambda_-}=\lr{1}$ if
$\sqrt{\Phie{\xi}}\leq\tau\leq \frac{3}{2}\sqrt{\Phie{\xi}}$, otherwise $\Lambda_+=\tau+c\sqrt{\Phie{\xi}}$ for some positive $c$ and
$\Lambda_-=\tau+c\sqrt{\Phie{\xi}}$ for some negative $c$ where $3/4\leq |c|\leq 5/4$.

With the same conditions but $l\geq 0$ being replaced by $ l<0$
we have
\[
F(\xi,\Lambda_{+})+F(\xi,\Lambda_{-})+\frac{\lr{\xi}^{-k}}{\lr{\xi}^{3B-}}
\]
inside the bracket of the second line of~\eqref{E:key-estimate-xi-4}.

When $|\xi|\leq 32\varepsilon^{-2}$ and $8B-2(l-k)>1$, we have the bound
\begin{equation}\label{E:key-estimate-xi-5}
\begin{split}
&{\bigg(}\int_{-\infty}^{\infty}\int_{-\infty}^{\infty}
\frac{\lr{\xi_1-\xi}^{2l}\lr{\xi_1}^{-2k}d\tau_1d\xi_1}{\lr{\tau_1+\Phie{\xi_1}}^{2B_1}
\lr{\tau_1-\tau\pm\sqrt{\Phie{\xi_1-\xi}}}^{2B_2}}{\bigg)^{1/2}}\\
&\leq C(B,k,l,\varepsilon).
\end{split}
\end{equation}

\end{lem}

\begin{proof}

By Lemma~\ref{L:s-1_s-2}, it suffices to prove that
\[
\begin{split}
& {\bigg(}\int_{-\infty}^{\infty}\frac{\lr{\xi_1-\xi}^{2l}\lr{\xi_1}^{-2k}d\xi_1}{\lr{\tau
\mp\sqrt{\Phie{\xi_1-\xi}}+\Phie{\xi_1}}^{2B}}{\bigg)^{1/2}}\\
&{\hskip 2cm}\leq C(B,\varepsilon,k,l)\max\{F(\xi,\Lambda_+),F(\xi,\Lambda_-)\}\,.
\end{split}
\]
By the decomposition
\[
\begin{split}
&\int_{-\infty}^{\infty}\frac{\lr{\xi_1-\xi}^{2l}\lr{\xi_1}^{-2k}d\xi_1}{\lr{\tau
\mp\sqrt{\Phie{\xi_1-\xi}}+\Phie{\xi_1}}^{2B}}\\
&=\int_{-\infty}^{0}\frac{\lr{\xi_1-\xi}^{2l}\lr{\xi_1}^{-2k}d\xi_1}{\lr{\tau
\mp\sqrt{\Phie{\xi_1-\xi}}+\Phie{\xi_1}}^{2B}}+\int_0^{\infty}
\frac{\lr{\xi_1-\xi}^{2l}\lr{\xi_1}^{-2k}d\xi_1}{\lr{\tau
\mp\sqrt{\Phie{\xi_1-\xi}}+\Phie{\xi_1}}^{2B}}\\
&=\int_{0}^{\infty}\frac{\lr{\xi_1+\xi}^{2l}\lr{\xi_1}^{-2k}d\xi_1}{\lr{\tau
\mp\sqrt{\Phie{\xi_1+\xi}}+\Phie{\xi_1}}^{2B}}+\int_0^{\infty}\frac{\lr{\xi_1-\xi}^{2l}\lr{\xi_1}^{-2k}d\xi_1}{\lr{\tau
\mp\sqrt{\Phie{\xi_1-\xi}}+\Phie{\xi_1}}^{2B}}\,,
\end{split}
\]
we can estimate the four integrals in the last line of the above equation.
From the reduction above, we may assume that $\xi\geq 0$. We can also see from
the following that $\Lambda_+$ comes from the integral with
$\tau+\sqrt{\Phie{\xi_1-\xi}}+\Phie{\xi_1}$ in the integrand while
$\Lambda_-$ comes from the integral with
$\tau-\sqrt{\Phie{\xi_1+\xi}}+\Phie{\xi_1}$ in the integrand.
Thus we will use $\Lambda$ instead of $\Lambda_+$ or
$\Lambda_-$ in the following proof for the simplicity of notation.

{\bf Part I.}
First, we estimate the integral
\begin{equation}\label{E:key-integral-2}
\int_{0}^{\infty}\frac{\lr{\xi_1-\xi}^{2l}\lr{\xi_1}^{-2k}}
{\lr{\tau+\sqrt{\Phie{\xi_1-\xi}}+\Phie{\xi_1}}^{2B}}\,d\xi_1\,.
\end{equation}
Before estimating the integral, note that for $l\geq 0\;,\;k<0$ and
$\xi_1\geq 0\;,\;\xi\geq 0$, the inequality
\begin{equation}\label{E:l_positive}
\lr{\xi_1-\xi}^{2l}\lr{\xi_1}^{-2k}\leq \lr{\xi_1+\xi}^{2l}\lr{\xi_1}^{-2k}
\end{equation}
always holds. But this is not true when $\xi_1$ is close to $\xi$ for $l< 0,\;k<0$ and $\xi_1\geq 0,\;\xi> 0$.
However we have
\begin{equation}\label{E:l_negative}
\begin{split}
&\int_{0}^{\infty}\frac{\lr{\xi_1-\xi}^{2l}\lr{\xi_1}^{-2k}}
{\lr{\tau+\sqrt{\Phie{\xi_1-\xi}}+\Phie{\xi_1}}^{2B}}\,d\xi_1\\
&\leq \int_{0}^{\infty}\frac{2\lr{\xi_1+\xi}^{2l}\lr{\xi_1}^{-2k}}
{\lr{\tau+\sqrt{\Phie{\xi_1-\xi}}+\Phie{\xi_1}}^{2B}}\,d\xi_1\\
&{\hskip 1.5cm}+\int_{\frac{1}{2}\xi}^{\frac{3}{2}\xi}\frac{\lr{\xi_1-\xi}^{2l}\lr{\xi_1}^{-2k}}
{\lr{\tau+\sqrt{\Phie{\xi_1-\xi}}+\Phie{\xi_1}}^{2B}}\,d\xi_1
\end{split}
\end{equation}
for $-1/2<l<0$.
This will be used to estimate~\eqref{E:key-integral-2}.
The last integral of~\eqref{E:l_negative} will be estimated in Lemma~\ref{L:around_xi}.
Hence we only need to estimate the first integral in the right hand side
of~\eqref{E:l_negative} where $l$ could be negative or non-negative.

For this purpose, we need to discuss the behavior of $\lr{\tau+\sqrt{\Phie{\xi_1-\xi}}+\Phie{\xi_1}}$. Define 
 $f_{\tau,\xi}(\xi_1)=\tau+\sqrt{\Phie{\xi_1-\xi}}+\Phie{\xi_1}$.
It is easy to see the function $f_{\tau,\xi}(\xi_1)$ has the absolute minimum.
Let $f_{\tau,\xi}(\xi_1)$ have the minimum $\tau+m$
when $\xi_1=A_m> 0$.  We note that $f_{\tau,\xi}(\xi_1)$ is not differentiable at
$\xi$ and there exists a constant $c<1$ such that
if $\xi<c$ then $A_m=\xi$. For $\xi\geq c$, we always have $A_m<\xi$ and
$A_m$ satisfies
\[
2A_m+4\varepsilon^2 A_m^3-\frac{1+2\varepsilon^2(\xi-A_m)^2}{\sqrt
{1+\varepsilon^2(\xi-A_m)^2}}=0.
\]
By comparing $1$ and $\varepsilon(\xi-A_m)$, we have
\begin{equation}\label{E:A_m_1}
\max\left\{\frac{1}{4},\frac{1}{\sqrt{2}}\varepsilon(\xi-A_m)\right\}
\leq A_m+2\varepsilon^2 A_m^3
\leq \max\left\{\frac{3}{2},\frac{3}{2}\varepsilon(\xi-A_m)\right\}.
\end{equation}
We also note that the $f_{\tau,\xi}(\xi_1)$ is convex as one can check that
\begin{equation}\label{E:second_derivative_f}
f''_{\tau,\xi}(\xi_1)=\frac{|\xi_1-\xi|(3\varepsilon^2+2\varepsilon^4(\xi_1-\xi)^2)}
{(1+\varepsilon^2(\xi_1-\xi)^2)^{3/2}}+2+12\varepsilon^2\xi_1^2> 0\;,\;\; \xi_1\neq\xi.
\end{equation}

\noindent{\bf Case 1.} $\xi\geq {32\varepsilon^{-2}}.$

If $\varepsilon(\xi-A_m)\leq 1$, then we have
$A_m+2\varepsilon^2 A_m^3\leq \frac{3}{2}$, hence
$A_m<\frac{3}{2}$ and $\varepsilon\xi\leq 1+\varepsilon A_m<5/2$.
Thus the relation $\varepsilon>32 (\xi\varepsilon)^{-1}$ leads to
$\varepsilon>\frac{64}{5}$ which contradicts  $\varepsilon\leq 1$ .
Therefore we only have to discuss
\[
\frac{1}{\sqrt{2}}\varepsilon(\xi-A_m)\leq A_m+2\varepsilon^2A_m^3\leq \frac{3}{2} \varepsilon(\xi-A_m).
\]
Since $A_m>0$  we have
\begin{equation}\label{E:A_m_2}
\min\left\{\frac{(1/\sqrt{2})\varepsilon}{1+(1/\sqrt{2})\varepsilon}\xi,(\frac{\xi}{4\sqrt{2}\varepsilon})^{1/3} \right\}
\leq A_m\leq \min\left\{\frac{(3/2)\varepsilon}{1+(3/2)\varepsilon}\xi,(\frac{3\xi}{4\varepsilon})^{1/3} \right\}.
\end{equation}
The restriction $\xi\geq 32\varepsilon^{-2}$ implies that
\begin{equation}\label{E:A_m}
A_m=c\left(\frac{\xi}{\varepsilon}\right)^{1/3}
\end{equation} where $1/2< c< 1$.
By Taylor's expansion, we have
\begin{equation}\label{E:sqr-Phie}
\sqrt{\Phie{\xi}}=\sqrt{\xi^2+\varepsilon^2\xi^4}
=\varepsilon\xi^2\sqrt{1+\frac{1}{\varepsilon^2\xi^2}}
=\varepsilon\xi^2+{\rm l.o.t.}
\end{equation}
and
\begin{equation}\label{E:m-1}
\begin{split}
m&=\sqrt{\Phie{\xi-A_m}}+\Phie{A_m}\\
&=\varepsilon(\xi-A_m)^2
+\frac{1}{2}(\xi-A_m)+\varepsilon^2A_m^4+A_m^2+{\rm l.o.t.}\\
&=\varepsilon\xi^2+(-2c+c^4)\varepsilon^{2/3}\xi^{4/3}+\frac{1}{2}\xi+{\rm l.o.t.}
\end{split}
\end{equation}
where l.o.t. means the lower order terms. We also note that the term $\frac{1}{2}\xi$ in the last line of equation~\eqref{E:m-1} can be absorbed by $\varepsilon^{2/3}\xi^{4/3}$ since $\xi\geq 32 \varepsilon^{-2}$.
From~\eqref{E:sqr-Phie} and~\eqref{E:m-1}, we have
\begin{equation}\label{E:m}
\frac{3}{4}\sqrt{\Phie{\xi}}\leq m \leq \sqrt{\Phie{\xi}}\,.
\end{equation}
To estimate~\eqref{E:key-integral-2}, we need to discuss different cases according to
different values of $\tau$, see Figure~\ref{F:Lambda}. Note that
\[
f_{\tau,\xi}(\xi_1)=\tau+m+(-m+\sqrt{\Phie{\xi_1-\xi}}+\Phie{\xi_1})\,.
\]

\begin{figure}
   \includegraphics[scale=.6]{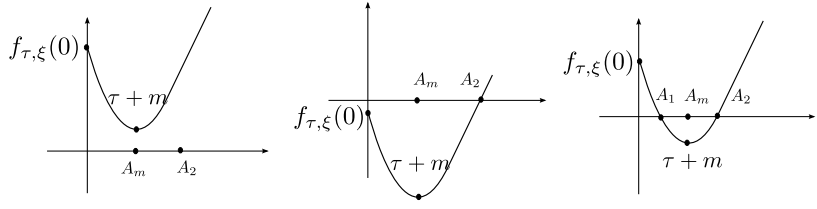}
    \caption{Graphs of $f_{\tau,\xi}(\xi_1)$ for different $\tau$.}
    \label{F:Lambda}
\end{figure}

\noindent{\bf Case 1a.} $f_{\tau,\xi}(A_m)=\tau+m\geq \sqrt{\Phie{\xi}}-m$. \par
In this case, we take $\Lambda=\tau+m$. The function $f_{\tau,\xi}(\xi_1)$ decreases
from $(0,\tau+\sqrt{\Phie{\xi}})$ to $(A_m,\tau+m)$, then increases from $(A_m,\tau+m)$ to
$(\infty,\infty)$.
We consider the decomposition
\begin{equation}\label{E:split_integration_1}
\int_{0}^{\infty}\frac{\lr{\xi_1+\xi}^{2l}\lr{\xi}^{-2k}}
{\lr{\tau+\sqrt{\Phie{\xi_1-\xi}}+\Phie{\xi_1}}^{2B}}\,d\xi_1=
\int_0^{5A_m}\cdots d\xi_1+\int_{5A_m}^{\infty}\cdots d\xi_1
\end{equation}
and estimate each term separately.
The condition $\Lambda\geq\sqrt{\Phie{\xi}}-m$ and~\eqref{E:A_m} imply
that $5A_m<C\varepsilon^{-1/2}\lr{\Lambda}^{1/4}$.
Hence~\eqref{E:A_m} and the direct estimate yield
\begin{equation}\label{E:integral_5A_m}
\begin{split}
&\int_{0}^{5A_m}\frac{\lr{\xi_1+\xi}^{2l}\lr{\xi_1}^{-2k}}
{\lr{\tau+\sqrt{\Phie{\xi_1-\xi}}+\Phie{\xi_1}}^{2B}}\,d\xi_1\\
&\quad\leq C(B)\frac{\lr{\xi}^{2l-2k}}
{\lr{\Lambda}^{2B}}\cdot\varepsilon^{-1/2}\lr{\Lambda}^{1/4}
\leq C(B,k,\varepsilon)\frac{\lr{\xi}^{2(l-k)}}
{\lr{\Lambda}^{2B-1/4}}
\end{split}
\end{equation}
provided $B>\frac{1}{8}$.

From estimates~\eqref{E:A_m} and~\eqref{E:m-1}, it is easy to check that
\begin{equation}\label{E:5A_m}
-m+\sqrt{\Phie{\xi_1-\xi}}+\frac{1}{4}\Phie{\xi_1}>0
\end{equation}
holds for $\xi_1>5A_m$. Therefore
\begin{equation}\label{E:reduced_integral}
\int_{5A_m}^{\infty}\frac{\lr{\xi_1+\xi}^{2l}\lr{\xi_1}^{-2k}}
{\lr{\tau+\sqrt{\Phie{\xi_1-\xi}}+\Phie{\xi_1}}^{2B}}\,d\xi_1
\leq\int_{5A_m}^{\infty}
\frac{\lr{\xi_1+\xi}^{2(l-k)}}{\lr{\Lambda+\frac{3}{4}\Phie{\xi_1}}^{2B}}\,d\xi_1\,.
\end{equation}
Let $\xi_1=\lr{\Lambda}^{1/4}\eta$.  If $\lr{\Lambda}\geq \lr{\xi}^4$, then
\[
\begin{split}
\int_{5A_m}^{\infty}
&\frac{\lr{\xi_1+\xi}^{2(l-k)}}{\lr{\Lambda+\frac{3}{4}\Phie{\xi_1}}^{2B}}\,d\xi_1\\
=&\frac{\lr{\Lambda}^{\frac{1}{4}(2l-2k)}}{\lr{\Lambda}^{2B}}\int_{5A_m}^{\infty}
\left(\frac{1+\xi+\xi_1}{\lr{\Lambda}^{1/4}}\right)^{2(l-k)}
\left(1+\frac{3}{4}\frac{\xi_1^2+\varepsilon^2\xi_1^4}{\lr{\Lambda}}\right)^{-2B}d\xi_1 \\
\leq& \frac{\lr{\Lambda}^{\frac{1}{4}(2l-2k)}}{\lr{\Lambda}^{2B}}\lr{\Lambda}^{1/4}\int_{0}^{\infty}
(1+\eta)^{2(l-k)}\left(1+\frac{3}{4}\Big(\frac{\eta^2}{\lr{\Lambda}^{1/2}}
+\varepsilon^2\eta^4\Big)\right)^{-2B}d\eta \\
\leq& C(B,k,l,\varepsilon) \frac{\lr{\Lambda}^{\frac{1}{4}(2l-2k)}}{\lr{\Lambda}^{2B-\frac{1}{4}}}
\end{split}
\]
provided $8B-2(l-k)>1$.
If $\lr{\Lambda}\leq \lr{\xi}^4$, then
\begin{equation}\label{E:large_Lambda}
\begin{split}
&\int_{5A_m}^{\infty}
\frac{\lr{\xi_1+\xi}^{2(l-k)}}{\lr{\Lambda+\frac{3}{4}\Phie{\xi_1}}^{2B}}\,d\xi_1\\
&=\frac{\lr{\xi}^{(2l-2k)}}{\lr{\Lambda}^{2B}}\int_{5A_m}^{\infty}
\left(\frac{1+\xi+\xi_1}{\lr{\xi}}\right)^{2(l-k)}
\left(1+\frac{3}{4}\frac{\xi_1^2+\varepsilon^2\xi_1^4}{\lr{\Lambda}}\right)^{-2B}d\xi_1 \\
&\leq \frac{\lr{\xi}^{(2l-2k)}}{\lr{\Lambda}^{2B}}\lr{\Lambda}^{1/4}\int_{0}^{\infty}
\left(1+\frac{\lr{\Lambda}^{1/4}}{\lr{\xi}}\eta\right)^{2(l-k)}
\left(1+\frac{3}{4}(\frac{\eta^2}{\lr{\Lambda}^{1/2}}+\varepsilon^2\eta^4)\right)^{-2B}d\eta \\
&\leq C(B,k,l,\varepsilon) \frac{\lr{\xi}^{(2l-2k)}}{\lr{\Lambda}^{2B-\frac{1}{4}}}
\end{split}
\end{equation}
provided $8B-2(l-k)>1$.

\noindent{\bf Case 1b.} $f_{\tau,\xi}(0)=\tau+\sqrt{\Phie{\xi}}\leq 0$. \par
In this case, we take $\Lambda=\tau+m$ again. Let $A_2$ be the number such that
$f_{\tau,\xi}(A_2)=0$. Then
\begin{equation}
\begin{split}
\varepsilon^2 A_2^4+A_2^2&=-\tau-\sqrt{\Phie{\xi-A_2}}\\
&\leq-\tau-\varepsilon(\xi-A_2)^2
 =(-\tau-\varepsilon\xi^2)+2\varepsilon\xi A_2+\varepsilon A_2^2\,,
\end{split}
\end{equation}
hence
\begin{equation}
\begin{split}
\varepsilon^2 A_2^4 &\leq\varepsilon^2 A_2^4+(1-\varepsilon)A_2^2
\leq (-\tau-\varepsilon\xi^2)+2\varepsilon\xi A_2\\
&\leq 2\max\{(-\tau-\varepsilon\xi^2),2\varepsilon\xi A_2\}.
\end{split}
\end{equation}
In case $\varepsilon^2 A_2^4\leq 2(-\tau-\varepsilon\xi^2)$, we have
$\varepsilon^2 A_2^4\leq 2(-\tau-\varepsilon\xi^2)\leq 2(-\tau-m)$.
When $\varepsilon^2 A_2^4
\leq 4\varepsilon\xi A_2$, we have $A_2\leq C\varepsilon^{-1/3}\xi^{1/3}$. From equations~\eqref{E:sqr-Phie} and~\eqref{E:m-1}, we note
that $|\Lambda|=|\tau+m|=|\tau+\sqrt{\Phie{\xi}}+m-\sqrt{\Phie{\xi}}|\geq|m-\sqrt{\Phie{\xi}}|
\geq C \varepsilon^{2/3}\xi^{4/3}$. Hence $A_2\leq C\varepsilon^{-1/2}\lr{\Lambda}^{1/4}$.
Therefore we can conclude
\begin{equation}\label{E:A_2}
A_2\leq C \varepsilon^{-1/2}\lr{\Lambda}^{1/4}.
\end{equation}
From the observation
\[
\max\{|(\xi-\xi_1)|,\varepsilon(\xi-\xi_1)^2\}\leq \sqrt{\Phie{\xi-\xi_1}}\leq 2
\max\{|(\xi-\xi_1)|,\varepsilon(\xi-\xi_1)^2\},
\]
$f_{\tau,\xi}(A_m)=\tau+m$ and
$f_{\tau,\xi}(A_2)=0$, it is easy to check that
\begin{equation}\label{E:5A_2}
\tau+\frac{1}{4}\Phie{\xi_1}+\sqrt{\Phie{\xi-\xi_1}}\geq |\tau+m|=|\Lambda|
\end{equation}
for $\xi_1\geq A_m+4(A_2-A_m)$. We split the integral
\[
\int_{0}^{\infty}\frac{\lr{\xi_1+\xi}^{2l}\lr{\xi_1}^{-2k}}
{\lr{\tau+\sqrt{\Phie{\xi_1-\xi}}+\Phie{\xi_1}}^{2B}}\,d\xi_1
\]
into two terms,
\begin{equation}\label{E:split_integration_1b_1}
\int_{0}^{5A_2-4A_m}\cdots d\xi_1+\int_{5A_2-4A_m}^{\infty}\cdots d\xi_1.
\end{equation}
Combining~\eqref{E:5A_2} and
following the same method as Case 1a. in estimating~~\eqref{E:reduced_integral}, we conclude that
\[
\int_{5A_2-4A_m}^{\infty}\frac{\lr{\xi_1+\xi}^{2l}\lr{\xi_1}^{-2k}}
{\lr{\tau+\sqrt{\Phie{\xi_1-\xi}}+\Phie{\xi_1}}^{2B}}\,d\xi_1
\leq C(B,l,k,\varepsilon)F(\xi,\Lambda_+).
\]

To estimate the first term of~\eqref{E:split_integration_1b_1}, we note that~\eqref{E:A_2} and
~\eqref{E:A_m} together imply that
the numerator of the integrand is bounded by
\begin{equation}
C(l,k,\varepsilon)\max\{\lr{\xi}^{2l-2k},
\lr{\Lambda}^{\frac{1}{4}(2l-2k)}\}.
\end{equation}
Hence it suffices to prove
\begin{equation}\label{E:split_integration_1b_2}
\int_{0}^{5A_2-4A_m}\frac{d\xi_1}{\lr{\tau+\sqrt{\Phie{\xi_1-\xi}}+\Phie{\xi_1}}^{2B}}
\leq C(B,\varepsilon)\frac{1}{\lr{\Lambda}^{(2B-\frac{1}{4})-}}.
\end{equation}
The left hand side of~\eqref{E:split_integration_1b_2} can be split into
\begin{equation}\label{E:split_integration_1b_3}
\begin{split}
&\int_0^{A_m}\cdots d\xi_1+\int_{A_m}^{A_2}\cdots d\xi_1+\int_{A_2}^{2A_2-A_m}\cdots d\xi_1+\\
&\hspace{1cm}\int_{2A_2-A_m}^{3A_2-2A_m}\cdots d\xi_1+\cdots+\int_{4A_2-3A_m}^{5A_2-4A_m}\cdots d\xi_1.
\end{split}
\end{equation}
First of all, we estimate the second term of~\eqref{E:split_integration_1b_3}, i.e., the integral
\begin{equation}\label{E:split_integration_1b_4}
\int_{A_m}^{A_2}\frac{d\xi_1}{\lr{\tau+\sqrt{\Phie{\xi_1-\xi}}+\Phie{\xi_1}}^{2B}}\,.
\end{equation}
It is natural to consider the partition
\begin{equation}\label{E:parition_A_2}
PA=\{a_0=A_m,a_1,\cdots,a_n, A_2\}
\end{equation}
where $a_i, i=1,\cdots,n,$ are given by
\begin{equation}
a_i=A_m+\frac{i}{i+1}(A_2-A_m)=\frac{1}{i+1}A_m+\frac{i}{i+1}A_2
\end{equation}
and $n$ is the integer closest to $|\Lambda|-1$. The inequality~\eqref{E:second_derivative_f}
implies that $f_{\tau,\xi}(\xi_1)$ is convex. Hence
\begin{equation}\label{E:A_i}
f_{\tau,\xi}(a_i)\leq\frac{1}{i+1}f_{\tau,\xi}(A_m)+\frac{i}{i+1}f_{\tau,\xi}(A_2)=\frac{1}{i+1}\Lambda<0.
\end{equation}
By inequality~\eqref{E:A_i} and $2B<1$, we have
\begin{equation}\label{E:A_m_to_A_2}
\begin{split}
\int_{A_m}^{a_n}&\frac{d\xi_1}{\lr{\tau+\sqrt{\Phie{\xi_1-\xi}}+\Phie{\xi_1}}^{2B}}=\int_{A_m}^{a_n}\frac{d\xi_1}
{(1+|f(\xi_1;\tau)|)^{2B}}\\
\leq&\sum_{i=1}^n\frac{(i+1)^{2B}}{\lr{\Lambda}^{2B}}(a_i-a_{i-1})
 \leq\frac{1}{\lr{\Lambda}^{2B}}\sum_{i=1}^n(i+1)(a_i-a_{i-1}) \\
=&\frac{1}{\lr{\Lambda}^{2B}}\times\{2(a_n-a_0)+(a_n-a_1)+\cdots+(a_n-a_{n-1})\}.
\end{split}
\end{equation}
Since $a_n<A_2$, we have
\[
a_n-a_i\leq A_2-a_i=\frac{1}{i+1}(A_2-A_m)<\frac{1}{i+1}A_2.
\]
Employing \eqref{E:A_2}, one can continue to estimate \eqref{E:A_m_to_A_2} as
\[
\eqref{E:A_m_to_A_2}\leq \frac{A_2}{\lr{\Lambda}^{2B}}\left(2+\frac{1}{2}+\cdots+\frac{1}{|\Lambda|}\right)\leq
\frac{C\varepsilon^{-1/2}\ln(|\Lambda|)}{\lr{\Lambda}^{2B-\frac{1}{4}}}=\frac{C\varepsilon^{-1/2}}{\lr{\Lambda}^{(2B-\frac{1}{4})-}}.
\]
Due to the choice of $a_n$, the trivial estimate and~\eqref{E:A_2} yield
\[
\int_{a_n}^{A_2}\frac{d\xi_1}{\lr{\tau+\xi_1^4+(\xi_1-\xi)^2}^{2B}}\leq 1\times(A_2-a_n)\leq\frac{1}{n}A_2=\frac{C\varepsilon^{-1/2}}{\lr{\Lambda}^{3/4}}.
\]
Thus the estimate of~\eqref{E:split_integration_1b_4} is completed.

Equation~\eqref{E:second_derivative_f} implies that the first derivative
 $f'_{\tau,\xi}(\xi_1)$ is strictly increasing.
Using this fact and $f_{\tau,\xi}(A_2)=0$, one can see that each integral from the third to the last terms of~\eqref{E:split_integration_1b_3} is smaller than the integration from $A_m$ to $A_2$. The estimate of the first term of~\eqref{E:split_integration_1b_3} is similar to the above.
Let $a_0=A_m$ and $a_i=\frac{1}{i+1}A_m\;,\;i=1,\cdots,|\Lambda|-1$. From the convexity of $f_{\tau,\xi}(\xi_1)$, we have
\[
f_{\tau,\xi}(a_i)\leq\frac{1}{i+1}f_{\tau,\xi}(A_m)+\frac{i}{i+1}f_{\tau,\xi}(0)\leq \frac{1}{i+1}\Lambda.
\]
Following the same arguments as before,
we conclude that the estimate holds for this term.

\noindent{\bf Case 1c.} $\tau+\sqrt{\Phie{\xi}}>0$ and $\tau+m<\sqrt{\Phie{\xi}}-m$. \par
We take $\lr{\Lambda}=\lr{1}$. In this case we have
\[
-\big(\sqrt{\Phie{\xi}}-m\big)<\tau+m<\sqrt{\Phie{\xi}}-m.
\]
The equation $f_{\tau,\xi}(\xi_1)=0$ may have two roots (see the third
graph of Figure~\ref{F:Lambda}), we let $A_2$ be the largest one.
From the above inequality,~\eqref{E:sqr-Phie} and~\eqref{E:m-1}, we have
\begin{equation}\label{E:A_2_less_5A_m}
A_2<5A_m.
\end{equation}
We split the integral into two terms as in~\eqref{E:split_integration_1} if $\tau+m\geq 0$,
otherwise it will be decomposed as in~\eqref{E:split_integration_1b_1}.  The second term
of each split can be treated as Case 1a or 1b respectively. The inequality~\eqref{E:A_2_less_5A_m}
implies that it suffices to estimate
\[
\int_{0}^{5A_m}\frac{\lr{\xi_1+\xi}^{2l}\lr{\xi_1}^{-2k}d\xi_1}{\lr{\tau+\sqrt{\Phie{\xi_1-\xi}}+\Phie{\xi_1}}^{2B}}.
\]
Choosing $\lr{\Lambda}=\lr{1}$, we have
\begin{equation}
\begin{split}
\int_{0}^{5A_m}&\frac{\lr{\xi_1+\xi}^{2l}\lr{\xi_1}^{-2k}d\xi_1}{\lr{\tau+\sqrt{\Phie{\xi_1-\xi}}+\Phie{\xi_1}}^{2B}}\\
\leq& C\lr{\xi}^{2l}\lr{(\xi/\varepsilon)^{1/3}}^{-2k}
\cdot \lr{(\xi/\varepsilon)^{1/3}}\\
\leq& C(k,\varepsilon){\lr{\xi}^{2(l-k)}}
\end{split}
\end{equation}
provided $k\leq -1/4$.

{\bf Case 2.} $0\leq\xi< 32\varepsilon^{-2}$.  \par
In this case, we need more effort in order to locate the value of $A_m$ and $m$.
Fortunately, this information is not necessary for our purpose. The boundedness of $\xi$ is sufficient for us to get the desired result.
By the discussion after~\eqref{E:l_positive} and~\eqref{E:l_negative},
we only need to estimate
\[
\int_{0}^{\infty}\frac{\lr{\xi_1+\xi}^{2l}\lr{\xi_1}^{-2k}}
{\lr{\tau+\sqrt{\Phie{\xi_1-\xi}}+\Phie{\xi_1}}^{2B}}\,d\xi_1
\]
since the last integral of~\eqref{E:l_negative} is bounded by $C(\varepsilon)$.

If $\tau>0$, then the integral is obviously  bounded
by a constant $C(B,k,l,\varepsilon)$ provided $8B-2(l-k)>1$.
If $-2\Phie{\xi}\leq\tau\leq 0$, the condition that $\xi$ is bounded 
implies that  $\tau$ is also bounded.
Hence the integral is bounded provided  $8B-2(l-k)>1$.
If $\tau<-2\Phie{\xi}$, we let $A_2$ be the number such that $f_{\tau,\xi}(A_2)=0$. Since $\xi$ is bounded,
$A_2\leq C\varepsilon^{-1/2}(-\tau)^{1/4}$. First we take
$\lr{\Lambda}=\lr{\tau+\sqrt{\Phie{\xi}}}$ instead of $\lr{\tau+m}$ and consider
\begin{equation}\label{E:xi-small}
\int_{0}^{\infty}\frac{\lr{\xi_1+\xi}^{2l}\lr{\xi_1}^{-2k}}
{\lr{\tau+\sqrt{\Phie{\xi_1-\xi}}+\Phie{\xi_1}}^{2B}}\,d\xi_1
=\int_0^{2A_2}\cdots d\xi_1 +\int_{2A_2}^{\infty}\cdots d\xi_1.
\end{equation}
Next we note that the condition $\tau<-2\Phie{\xi}$ implies $\lr{\tau}\approx\lr{\Lambda}=\lr{\tau+\sqrt{\Phie{\xi}}}$. Hence
$A_2\leq C\varepsilon^{-1/2}\lr{\Gamma}^{1/4}$.
Using~\eqref{E:second_derivative_f} and the fact $f_{\tau,\xi}(A_2)=0$, one can
prove as in Case 1b that the both integrals in the right hand side of
~\eqref{E:xi-small} are bounded by
\[
C(B,k,l,\varepsilon)\frac{\lr{\Gamma}^{\frac{1}{2}(l-k)}}
{\lr{\Gamma}^{(2B-\frac{1}{4})-}}
\]
if $8B-2(l-k)>1$. The condition $8B-2(l-k)>1$ also implies the
fraction above is finite.

{\bf Part II.}
Next, we estimate the integral
\begin{equation}\label{E:part-ii}
\int_{0}^{\infty}\frac{\lr{\xi_1+\xi}^{2l}\lr{\xi_1}^{-2k}}
{\lr{\tau-\sqrt{\Phie{\xi_1+\xi}}+\Phie{\xi_1}}^{2B}}\,d\xi_1.
\end{equation}
Note that the numerator of the integrand is
$\lr{\xi_1+\xi}^{2l}\lr{\xi_1}^{-2k}$.
The estimates are similar to those in Part I. First, we define $f_{\tau,\xi}(\xi_1)=\tau-\sqrt{\Phie{\xi_1+\xi}}+\Phie{\xi_1}$
and discuss the cases $\xi\geq 32\varepsilon^{-2}$
and $\xi<32\varepsilon^{-2}$ separately.
A straightforward differentiation yields
\begin{equation}\label{E:f_2}
f'_{\tau,\xi}(\xi_1)=2\xi_1+4\varepsilon^2\xi_1^3-\frac{1+2\varepsilon^2(\xi_1+\xi)^2}{\sqrt{1+\varepsilon^2(\xi_1+\xi)^2}},
\end{equation}
thus there exists a number $A_m$ such that $f_{\tau,\xi}(\xi_1)$ has the absolute
minimum $m$ at $A_m$. 

When $\xi> 32\varepsilon^{-2}$ we have $A_m=c(\frac{\xi}{\varepsilon})^{1/3}$
as one can find from~\eqref{E:f_2}. We also note that
\[
\sqrt{\Phie{\xi}}\leq m\leq \frac{5}{4}\sqrt{\Phie{\xi}}.
\]

Next we claim that $f_{\tau,\xi}(\xi_1)$ is a convex function.
One finds
\[
f''_{\tau,\xi}(\xi_1)=2+12\varepsilon^2\xi_1^2-
\frac{\varepsilon^2(\xi_1+\xi)[3+2\varepsilon^2(\xi_1+\xi)^2]}{(1+\varepsilon^2(\xi_1+\xi)^2)^{3/2}}.
\]
Let $a=\varepsilon(\xi_1+\xi)$ and rewrite
\[
\frac{\varepsilon^2(\xi_1+\xi)[3+2\varepsilon^2(\xi_1+\xi)^2]}{(1+\varepsilon^2(\xi_1+\xi)^2)^{3/2}}=
\varepsilon\frac{3a+2a^3}{(1+a^2)^{3/2}}=g(a).
\]
It is easy to see $g(a)$, $a\geq\varepsilon\xi$, is monotone increasing to $2\varepsilon$. Hence we have
\[
f''_{\tau,\xi}(\xi_1)> 0.
\]
With this information, we discuss the three cases $\tau+m\geq m-\sqrt{\Phie{\xi}}$,
$\tau-\sqrt{\Phie{\xi}}$ and $\sqrt{\Phie{\xi}}<\tau<2m-\sqrt{\Phie{\xi}}$ respectively.  Then we can prove the result by the same method as in Case 1 of  Part I.

When $0\leq \xi \leq 32\varepsilon^{-2}$ we can again get the desired bound
for~\eqref{E:part-ii} by the same
method as in Case 2 of Part I.

{\bf Part III.}
Now we estimate the integral
\begin{equation}
\int_{0}^{\infty}\frac{\lr{\xi_1+\xi}^{2l}\lr{\xi_1}^{-2k}}
{\lr{\tau+\sqrt{\Phie{\xi_1+\xi}}+\Phie{\xi_1}}^{2B}}\,d\xi_1.
\end{equation}
In this case, we define $f_{\tau,\xi}(\xi_1)=\tau+\sqrt{\Phie{\xi_1+\xi}}+\Phie{\xi_1}$ while $A_m$ and $m$ are defined as before.
The function $f_{\tau,\xi}(\xi_1)$ is  strictly increasing from $0$. Hence $A_m=0$ and $m=\sqrt{\Phie{\xi}}$.
$f_{\tau,\xi}(\xi_1)$ is convex since
\[
f''_{\tau,\xi}(\xi_1)=2+12\varepsilon^2\xi_1^2+
\frac{\varepsilon^2(\xi_1+\xi)[3+2\varepsilon^2(\xi_1+\xi)^2]}{(1+\varepsilon^2(\xi_1+\xi)^2)^{3/2}}>2.
\]
We omit the details since it is easy to follow the same method of Part I to obtain the desired result.

{\bf Part IV.}
Finally, we estimate the integral
\begin{equation}
\int_{0}^{\infty}\frac{\lr{\xi_1-\xi}^{2l}\lr{\xi_1}^{-2k}}
{\lr{\tau-\sqrt{\Phie{\xi_1-\xi}}+\Phie{\xi_1}}^{2B}}\,d\xi_1.
\end{equation}
By Lemma~\ref{L:around_xi}
and the discussion after~\eqref{E:l_positive} and~\eqref{E:l_negative},
we only have to estimate
\[
\int_{0}^{\infty}\frac{\lr{\xi_1+\xi}^{2l}\lr{\xi_1}^{-2k}}
{\lr{\tau-\sqrt{\Phie{\xi_1-\xi}}+\Phie{\xi_1}}^{2B}}\,d\xi_1.
\]
We define $f_{\tau,\xi}(\xi_1)=\tau-\sqrt{\Phie{\xi_1-\xi}}+\Phie{\xi_1}$.
We note that $f_{\tau,\xi}(\xi_1)$ is not differentiable at $\xi$. From
\[
f'_{\tau,\xi}(\xi_1)=2\xi_1+4\varepsilon^2\xi_1^3-
\frac{(\xi_1-\xi)+2\varepsilon^2(\xi_1-\xi)^3}{\sqrt{(\xi_1-\xi)^2
+\varepsilon^2(\xi_1-\xi)^4}}>0,\quad \xi_1\neq\xi
\]
we know $f$ is strictly increasing from $0$ and moreover, $f_{\tau,\xi}(\xi_1)$ is convex since
\[
f''_{\tau,\xi}(\xi_1)=2+12\varepsilon^2\xi_1^2-
\frac{\varepsilon^2|\xi_1-\xi|[3+2\varepsilon^2(\xi_1-\xi)^2]}
{(1+\varepsilon^2(\xi_1-\xi)^2)^{3/2}}>0,\quad \xi_1\neq\xi.
\]
Therefore we can estimate the integral by the method developed in Part I.

\end{proof}

Finally we prove the supplemental lemma.

\begin{lem}\label{L:around_xi}
If  $k<0\;,\; l<0$, $0<B<1/2$, and $\xi>32\varepsilon^{-2}$. Then
\begin{equation}
\int_{\frac{1}{2}\xi}^{\frac{3}{2}\xi}\frac{\lr{\xi_1-\xi}^{2l}\lr{\xi_1}^{-2k}}
{\lr{\tau\pm\sqrt{\Phie{\xi_1-\xi}}+\Phie{\xi_1}}^{2B}}\,d\xi_1
\leq C(\varepsilon)\frac{\lr{\xi}^{-2k}}{\lr{\xi}^{6B-}}.
\end{equation}
\end{lem}
\begin{proof}
We observe that for $l<0$
\begin{equation}\label{E:around-xi-1}
\begin{split}
&\int_{\frac{1}{2}\xi}^{\frac{3}{2}\xi}\frac{\lr{\xi_1-\xi}^{2l}\lr{\xi_1}^{-2k}}
{\lr{\tau\pm\sqrt{\Phie{\xi_1-\xi}}+\Phie{\xi_1}}^{2B}}\,d\xi_1\\
& \leq C\lr{\xi}^{-2k}\int_{\frac{1}{2}\xi}^{\frac{3}{2}\xi}\frac{1}
{\lr{\tau\pm\sqrt{\Phie{\xi_1-\xi}}+\Phie{\xi_1}}^{2B}}\,d\xi_1.
\end{split}
\end{equation}

Let $\tau=-\Phie{c_{\tau}\xi}$ and $b_{\tau}$ be the number such that
\begin{equation}\label{E:b-tau}
\tau\pm\sqrt{\Phie{b_{\tau}\xi-\xi}}+\Phie{b_{\tau}\xi}=-\Phie{c_{\tau}\xi}\pm\sqrt{\Phie{b_{\tau}\xi-\xi}}+\Phie{b_{\tau}\xi}=0. 
\end{equation}
It is clear that we only have to consider the case where $b_\tau\in[\frac{1}{2}\xi,\frac{3}{2}\xi]$.
Let $I_b$ be the interval with $b_\tau$ as the center with radius $1$. Its intersection with
$[\frac{1}{2}\xi,\frac{3}{2}\xi]$ is again denoted by $I_b$.
From the relation~\eqref{E:b-tau}, we know that
\[
C(\varepsilon)\xi^3\leq \lr{\tau\pm\sqrt{\Phie{\xi_1-\xi}}+\Phie{\xi_1}}
\]
if $\xi_1\in [\frac{1}{2}\xi,\frac{3}{2}\xi]\setminus I_b$.  In fact the value of
$\lr{\tau\pm\sqrt{\Phie{\xi_1-\xi}}+\Phie{\xi_1}}$ increases to $C(\varepsilon)\xi^4$ when
$\xi_1$ tends to $\frac{1}{2}\xi$ or $\frac{3}{2}\xi$.
Since the length of interval
$[\frac{1}{2}\xi,\frac{3}{2}\xi]\setminus I_b$ is of order $\xi$. We can estimate
the integral
\[
\int_{[\frac{1}{2}\xi,\frac{3}{2}\xi]\setminus I_b}\frac{d\xi_1}{\lr{\tau\pm\sqrt{\Phie{\xi_1-\xi}}+\Phie{\xi_1}}^{2B}}
\]
by considering the sum of partitions as in the previous lemma and get the desired bound.    

The estimate of integration over $I_b$ is similar by noting that the interval $I_b$ has length not bigger than 2
and the variation of
$|\tau\pm\sqrt{\Phie{\xi_1-\xi}}+\Phie{\xi_1}|$ on $I_b$ is of order $C(\varepsilon)\lr{\xi}^3$.

\end{proof}

\section*{Acknowledgement}
The research of the first author is supported by National Science
Council Grant NSC100-2115-M-007-009-MY2. 
He would also like to express his gratitude to Y. Tsutsumi and K. Nakanishi  for valuable
comments on the draft.
The research of the second author is supported by National Science
Council Grant NSC101-2115-M-009-008-MY2. The research of the third author is supported by NSF grant 1160981.
The authors would like to thank anonymous referee for helpful comments and suggestions
which improve the presentation of paper significantly.

\end{document}